\documentclass{amsart}
\usepackage[utf8]{inputenc}
\usepackage{amssymb}
\usepackage{mathrsfs}
\usepackage{todonotes}

\usepackage[bottom=3cm, top=3cm, left=3.5cm, right=3.5cm]{geometry}

\usepackage{xcolor}
\usepackage[bookmarks=true]{hyperref}
\hypersetup{
    colorlinks,
    linkcolor={red!50!black},
    citecolor={blue!50!black},
    urlcolor={blue!80!black},
}

\usepackage{bbm}
\usepackage{upgreek}
\usepackage{mathtools}

\usepackage[shortlabels]{enumitem}
\usepackage[capitalise]{cleveref}

\usepackage{autonum} %to number only \eqref'ed equations in a cref compatible way

 \title[Tubes in sub-Riemannian geometry]{Tubes in sub-Riemannian geometry and a Weyl's invariance result for curves in the Heisenberg groups}

\author{Tania Bossio}
\address{Tania Bossio:  University of Fribourg, Chemin du Musée 23, 1700 Fribourg, Switzerland}
\email{\href{mailto:tania.bossio@unifr.ch}{tania.bossio@unifr.ch}}

\author{Luca Rizzi}
\address{Luca Rizzi: SISSA, via Bonomea 265, 34136 Trieste, Italy}
\email{\href{mailto:lrizzi@sissa.it}{lrizzi@sissa.it}}

\author{Tommaso Rossi}
\address{Tommaso Rossi: Università degli Studi dell'Aquila, Via Vetoio, 67100 L'Aquila AQ, Italy}
\email{\href{tommaso.rossi1@univaq.it}{tommaso.rossi1@univaq.it}}

\date{\today}

\newcommand{\sr}{sub-Rieman\-nian }
\newcommand{\di}{{\sf d}}

\newcommand{\de}{\,{\rm d}}
\newcommand{\diverg}{{\rm div}}
\newcommand{\N}{\mathbb N}

\newcommand{\R}{\mathbb{R}}

\newcommand{\ann}{A}

\newcommand{\D}{\mathcal{D}}

\newcommand{\unitexp}{E^+}

\newcommand\restr[2]{{% we make the whole thing an ordinary symbol
  \left.\kern-\nulldelimiterspace % automatically resize the bar with \right
  #1 % the function
  \right|_{#2} % this is the delimiter
  }}

\newcommand{\ds}{\delta}                               % distanza da S
\theoremstyle{plain}
\newtheorem{thm}{Theorem}[section]
\newtheorem{cor}[thm]{Corollary}
\newtheorem{lem}[thm]{Lemma}
\newtheorem*{lem*}{Lemma}
\newtheorem{prop}[thm]{Proposition}

\theoremstyle{definition}
\newtheorem{defn}[thm]{Definition}

\theoremstyle{remark}
\newtheorem{rmk}[thm]{Remark}
\newtheorem{example}[thm]{Example}

\crefname{thm}{Theorem}{Theorems}						% ref name for cleverref package
\crefname{theoremintro}{Theorem}{Theorems}				% ref name for cleverref package

\begin{document}
\begin{abstract}
The purpose of the paper is threefold: first, we prove optimal regularity results for the distance from $C^k$ submanifolds of general rank-varying sub-Riemannian structures. Then, we study the asymptotics of the volume of tubular neighbourhoods around such submanifolds. Finally, for the case of curves in the Heisenberg groups, we prove a Weyl's invariance result: the volume of small tubes around curves does not depend on the way the curve is isometrically embedded, but only on its Reeb angle. The proof does not need the computation of the actual volume of the tube, and it is new even for the three-dimensional Heisenberg group.
\end{abstract}

\maketitle
%\tableofcontents

\section{Introduction}

Let $S$ be a closed submanifold of codimension $m$ in the $n$-dimensional Euclidean space $\R^n$. In \cite{Weyl-Tubes}, Weyl derived a formula for the volume of a tube of small radius $r$ around $S$:
\begin{equation}\label{eq:tube}
V(r) = \frac{\pi^{m/2}}{\Gamma\left(\tfrac{n}{2}+1\right)} \sum_{\substack{e=0 \\e \text{ even}}}^{n-m}\frac{k_{e}(S) r^{m+e}}{(m+2)(m+4)\cdots (m+e)}.
\end{equation}
There are two noteworthy features of this formula. Firstly, its \emph{regularity}: the function $V(r)$ is a polynomial, and not some more complicated function. In particular, it is real-analytic, even if $S$ is only supposed to be smooth, and only the terms $r^{m+e}$, with $e$ even, appear in the polynomial. Secondly, its \emph{invariant nature}: the coefficients $k_{e}(S)$ can be written in terms of the intrinsic curvature tensor of $S$, and thus formula \eqref{eq:tube} does not depend on the way $S$ is embedded in $\R^n$, but only on its inner metric structure induced by $\R^n$.

Inspired by Weyl's results, we study the volume of tubes in more general metric spaces, and more precisely in (sub-)Riemannian manifolds. This research is motivated by the recent works in this area \cite{F-StainerBalls,Balogh-Steiner,R-tub-neigh}, which focused on the case of tubes around  \emph{hypersurfaces} embedded in the Heisenberg group, and \cite{BB-Steiner3D}, for tubes around hypersurfaces in three-dimensional contact structures. (In this case, the corresponding volume asymtptotics are known as \emph{Steiner's formulas}.) See also \cite{BBL-balls} where, with a different spirit, the authors study the volume of small balls in three-dimensional contact structures.

We develop the theory for $C^k$ non-characteristic submanifolds of arbitrary codimension embedded in a general, possibly rank-varying, sub-Riemannian manifold. In the specific case of the Heisenberg groups, we prove that the volume of small tubes around curves does not depend on the way the curve is isometrically embedded, but only on its Reeb angle. This result extends the classical Euclidean one due to Weyl. Its proof is based on a symmetry argument of independent interest, new even in the Euclidean setting.

\subsection{Regularity of the distance and tubular neighbourhoods}

%We are, in particular, inspired by the former approach, where the authors were able to express the tube coefficients in terms of \emph{iterated divergences}.

The study of tubes begins with the definition of what a tube is, and instrumental to this is the study of the regularity of the distance from a submanifold. This constitutes our first result, which continues the research line initiated in \cite{AF-normal,AF-Hessian,F-StainerBalls,R-tub-neigh} for $C^2$ hypersurfaces or submanifolds of the Heisenberg groups and \cite{AFM-step2} for $C^\infty$ hypersurfaces of some special  step $2$ Carnot groups. Our contribution extends all previous results to $C^k$ submanifolds of arbitrary codimension, embedded in general rank-varying sub-Riemannian manifolds, with optimal regularity in the $C^k$ class.

In the following, $\delta: M \to \R$ is the distance from an embedded submanifold $S\subset M$, $\pi:T^*M \to M$ is the cotangent bundle projection, $AS \subset T^*M$ is the bundle of covectors that vanish on $TS$ (the so-called \emph{annihilator bundle}), $E:AS \to M$ is the restriction of the sub-Riemannian exponential map to $AS$, $H  \in C^\infty(T^*M)$ denotes the sub-Riemannian Hamiltonian, while $\nabla$ represents the horizontal gradient of a $C^1$ function. We refer to \cref{sec:preliminaries} for precise definitions. The next result corresponds to \cref{thm:tubolarneigh}.

\begin{thm}[Sub-Riemannian tubular neighbourhoods and regularity of the distance] 
\label{thm:tubolarneigh-intro}
    Let $M$ be a sub-Riemannian manifold (smooth, without boundary, complete) and let $S \subset M$ be a non-characteristic submanifold of codimension $m\geq 1$ and of class $C^k$, with $k\geq 2$ (without boundary). Then, there exists a continuous function $\varepsilon : S \to \R_{>0}$, such that, letting
    \begin{equation}\label{eq:defV-intro}
V:=\left\{\lambda \in AS \,\Big|\, \sqrt{2H(\lambda)} < \varepsilon(\pi(\lambda))\right\},
    \end{equation}
    the following statements hold:
    \begin{enumerate}[(i)]
        \item \label{item:tubolarneigh1-intro} The restriction of the normal exponential map $E:V\to U:=E(V)$ is a $C^{k-1}$ diffeomorphism;
        \item \label{item:tubolarneigh2-intro} For all $p=E(\lambda) \in U$ there exists a unique minimizing geodesic $\gamma:[0,1]\to M$ from $S$, which is normal, and it is given by $\gamma_t = E(t\lambda)$. In particular, on $V$ it holds
        \begin{equation}
                \ds \circ E = \sqrt{2H};
        \end{equation}
        \item \label{item:tubolarneigh3-intro}  $\delta \in C^{k}(U\setminus S)$, with
                \begin{equation}
        \label{eq:eikonal-intro}
            \|\nabla \ds\|=1,\qquad\text{on } U\setminus S;
        \end{equation}
\item \label{item:tubolarneigh4-intro} $\delta^2\in C^{k}(U)$;
\item \label{item:tubolarneigh5-intro} Let $X,Y$ be smooth (or real-analytic, if $M$ is real-analytic) vector fields. Then the functions $\delta$, $X\delta$, $YX\delta$ are smooth (or real-analytic) along any minimizing geodesic from $S$ contained in $U\setminus S$.
 \end{enumerate}
\end{thm}
\begin{rmk}[Necessity of assumptions]
Without adding further hypotheses, the assumption $k\geq2$ is necessary, already in the Euclidean case, see \cite{KrantzParks-distance,MR749908}. The non-characteristic assumption is also crucial: $\delta$ is not even Lipschitz in charts at characteristic points, see \cite{ACS_reg}.
\end{rmk}
% \begin{rmk}[Related results]\todo{L: magari mettere dopo i remark, non in remark?}
% \end{rmk}
\begin{rmk}[From $C^{k-1}$ to $C^k$]
 \cref{item:tubolarneigh1-intro,,item:tubolarneigh2-intro} immediately imply \cref{item:tubolarneigh3-intro,,item:tubolarneigh4-intro} with $C^{k-1}$ regularity. The improvement to $C^k$ regularity, $k\geq 2$, requires more work. This unexpected gain of regularity was observed in the Euclidean case in \cite{KrantzParks-distance,MR749908}.
\end{rmk}

\begin{rmk}[Regularity along geodesics]
\cref{item:tubolarneigh5-intro} states that $\delta$ and its derivatives up to order two are smooth (or real-analytic) along minimizing geodesics from the submanifold, while the latter is only $C^2$. This is a consequence of the action of the underlying Hamiltonian flow.
\end{rmk}

A version of \cref{thm:tubolarneigh-intro} appeared in \cite{Rossi22} when $S$ is closed and smooth ($k=\infty$), building on the codimension $1$ case treated in \cite{MR4117982}. With respect to those references, \cref{item:tubolarneigh3-intro,,item:tubolarneigh4-intro,,item:tubolarneigh5-intro} are new and require new arguments for $2\leq k<+\infty$.
In \cite{ACS_reg}, among other results, a local analogue of \cref{item:tubolarneigh3,item:tubolarneigh4} is obtained for smooth hypersurfaces ($k=\infty$, $m=1$), with different techniques. A version of \cref{thm:tubolarneigh-intro} for the Heisenberg groups was proven in \cite{R-tub-neigh}, using the explicit knowledge of minimizing geodesics.

For \emph{two-sided} hypersurfaces, it is customary to define an associated \emph{signed distance} $\delta_{\mathrm{sign}}: M \to \R$, see \cref{sec:twosided}, enjoying better regularity properties. We record in this case a variant of \cref{thm:tubolarneigh-intro},  corresponding to \cref{cor:twoside}.

\begin{cor}[The two-sided case]\label{cor:twoside-intro}
In the same setting of \cref{thm:tubolarneigh-intro}, assuming furthermore that $S$ is a two-sided non-characteristic hypersurface, \cref{item:tubolarneigh3-intro,item:tubolarneigh4-intro,item:tubolarneigh5-intro} hold on up to $S$ (i.e.\ on the whole $U$), replacing $\delta$ with $\delta_{\mathrm{sign}}$.
\end{cor}

The function $\varepsilon: S\to \R$ of \cref{thm:tubolarneigh-intro} may tend to zero when $S$ is non-compact or $\mathrm{cl}(S)\setminus S$ is non-empty. If $S$ is bounded and \emph{extendible} (roughly speaking, it is a subset of a larger submanifold, see \cref{def:extendible} for a precise definition), the tubular neighbourhood of \cref{thm:tubolarneigh-intro} can be chosen to be uniform. See \cref{thm:uniformneigh}.

\begin{thm}[Uniform tubular neighbourhoods]\label{thm:uniformneigh-intro}
In the setting of \cref{thm:tubolarneigh-intro}, assume furthermore that $S$ is bounded and satisfies the extendibility property. Then, there exists $r_0=r_0(S)>0$ such that \cref{item:tubolarneigh1-intro,item:tubolarneigh2-intro,item:tubolarneigh3,item:tubolarneigh4-intro,item:tubolarneigh5-intro} hold for
\begin{equation}\label{eq:Vuniform-intro}
V:=\left\{ \lambda \in AS \,\Big|\, \sqrt{2H(\lambda)}< r_0 \right\}.
\end{equation}
\end{thm}
\begin{rmk}\label{rmk:injrad-intro}
If there is $V$ of the form \eqref{eq:Vuniform-intro} such that \cref{item:tubolarneigh1-intro,item:tubolarneigh2-intro,item:tubolarneigh3-intro,item:tubolarneigh4-intro,item:tubolarneigh5-intro} of \cref{thm:tubolarneigh-intro} hold, we say that $S$ has \emph{positive injectivity radius}. Note that if $\mathrm{cl}(S)\supsetneq S$ in $M$, then $r_0$ can be smaller of the supremum of the $r>0$ s.t.\ $E$ is a diffeomorphism on $V_r=\{\lambda \in AS \mid \sqrt{2H(\lambda)}<r\}$.
\end{rmk}

For submanifolds with positive injectivity radius we can finally define tubes. %See Figure \ref{fig:fig}.

\begin{figure}
% \label{fig:fig}
        %\centering
%\includegraphics[scale=0.7]{Disegni/Tubo_tubo.eps} 
 \includegraphics[scale=0.8]{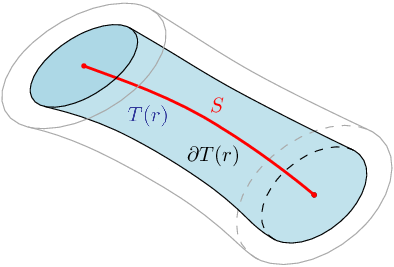} \qquad\qquad\qquad
 \includegraphics[scale=0.8]{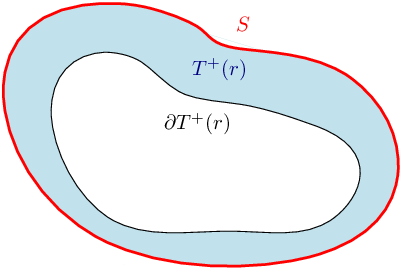} 
	\caption{Tubes and half-tubes.}
\end{figure}
\begin{defn}[Tubes and half-tubes]
\label{def:tube}
Let $S\subset M$ be a $C^2$ non-characteristic embedded submanifold of codimension $m\geq 1$, bounded and with the extendibility property. Let $r_0=r_0(S)>0$ be its injectivity radius. The \emph{tube of radius} $r\in (0,r_0)$ is the set
\begin{equation}
T(r):= E\left(\{\lambda \in A S \mid \sqrt{2H(\lambda)}<r\}\right).
\end{equation}
The \emph{tubular hypersurface at radius} $r \in (0,r_0)$ is the set 
\begin{equation}
\partial T(r):= E\left(\{\lambda \in A S \mid \sqrt{2H(\lambda)}=r\}\right).
\end{equation}
If $S$ is a two-sided hypersurface, the \emph{half-tubes} of radius $r \in (0,r_0)$ are the sets
    \begin{equation} \label{eq:def_omega_prime-intro}
    T^{\pm}(r):= E\left(\{\lambda \in A^{\pm} S \mid \sqrt{2H(\lambda)}<r\}\right).
    \end{equation}
In this case, the \emph{tubular hypersurface at radius} $r \in (0,r_0)$ is the set 
\begin{equation}
\partial T^{\pm}(r):= E\left(\{\lambda \in A^{\pm} S \mid \sqrt{2H(\lambda)}=r\}\right).
\end{equation}
\end{defn}
If $S$ is a closed manifold embedded in $M$, then $T(r)$ coincides with the set of points at distance smaller than $r$ from $S$. In general, however, $T(r) \subsetneq\{\delta < r\}$.

\subsection{The volume of tubes}

Our second main result is on the asymptotics of the volume of tubes. We were inspired, in particular, by the approach of \cite{Balogh-Steiner}, who introduced the concept of \emph{iterated divergences}, to account for the lack of a tensorial calculus well-adapted to the sub-Riemannian setting, akin to the one used in deriving the classical Weyl's tube formula \cite{Gray-Tubes}. Here, we improve and extend this approach to general codimension.

\begin{defn}[Iterated divergences]
\label{def:diverg}
Let $\omega$ be a smooth positive density on a smooth manifold, and $X$ be a vector field. The \emph{iterated divergences} are the functions $\diverg^k_\omega(X)$ such that
\begin{equation}
\mathscr{L}^k_X \omega = \diverg^k_\omega(X) \omega,\qquad k\in \N,
\end{equation}
provided that the Lie derivatives exists.
\end{defn}

Note that $\diverg^0_\omega(X) = 1$, and $\diverg^1_\omega(X) = \diverg_\omega(X)$ is the classical divergence of $X$. If they exist, the iterated divergences satisfy the following recursive relations:
\begin{equation}\label{eq:recursive_relation}
\diverg^{k+1}_\omega(X) = \diverg_\omega(\diverg^{k}_\omega(X)X)  = \diverg_\omega(X)\diverg^{k}_\omega(X) + X(\diverg^{k}_\omega(X)).
\end{equation}

In the next result, $A^1S = AS \cap \{2H = 1\}$ is the unit annihilator bundle, and $E_r: A^1S\to M$ is the exponential map at time $r$, namely $E_r(\lambda) = E(r\lambda)$. See \cref{thm:weyl}.
\begin{thm}[Weyl's tube formula]
\label{thm:weyl-intro}
 Let $M$ be a smooth (or real-analytic) \sr manifold, equipped with a smooth (or real-analytic) measure $\mu$. Let $S \subset M$ be a bounded non-characteristic embedded submanifold of codimension $m\geq 1$, of class $C^2$ (without boundary) and with the extendibility property. Let $r_0=r_0(S)>0$ be its injectivity radius. Then the volume of the tube $r\mapsto \mu(T(r))$ is smooth (or real-analytic) on 
    $[0,r_0)$. Furthermore, there exists a continuous density $\sigma_m$ on $A^1S$, defined by
    \begin{equation}\label{eq:sigmam-intro}
        \sigma_m := \lim_{r \to 0} \frac{E^*_r(\iota_{\nabla\delta} \mu)}{r^{m-1}} ,
    \end{equation}
    and continuous functions $w_{m}^{(j)}:A^1S \to \R$ defined for $j\in \N$ by
    \begin{equation}\label{eq:wmj-intro}
        w_{m}^{(j)} := \lim_{r\to 0} \diverg^{j}_{\mu/\delta^{m-1}}(\nabla\delta) \circ E_r,
    \end{equation}
    such that $\mu(T(r))$ has the following Taylor expansion at $r=0$:
    \begin{equation}
    \label{eq:expansion-intro}
         \mu(T(r))= \sum_{\substack{k\geq m \\ k-m \text{ even}}}  \frac{1}{k(k-m)!}\left(\int_{A^1 S} w_m^{(k-m)} \,\de\sigma_m \right) r^k,
    \end{equation}
    where $\de\sigma_m$ denotes the measure induced by $\sigma_m$.
    \end{thm}

\begin{rmk}[On the regularity]
In this generality, one cannot expect $\mu(T(r))$ to be a polynomial in $r$, as it happens for Euclidean tubes. However by the above theorem, as soon as the ambient structure is smooth (resp.\ real-analytic), the volume of tubes is smooth (resp.\ real-analytic) even if the corresponding submanifold $S$ is just $C^2$. Note that the real-analyticity of the volume of the tube around smooth hypersurfaces ($k=\infty$, $m=1$) in the three-dimensional Heisenberg group was obtained in \cite{Balogh-Steiner} as a consequence of the explicit computation of all coefficients in the Taylor expansion. Our proof uses different ideas.
\end{rmk}

If $S$ is a two-sided hypersurface, and one considers its (positive) half-tube $T^+(r)$, the density $\sigma_1$ and the functions $w_1^{(j)}$ of \cref{thm:weyl-intro} can be identified, respectively, with the sub-Riemannian perimeter measure $\sigma$ and the iterated divergences $\diverg^j_\mu(\nabla\delta_{\mathrm{sign}})$ on $S$. In this setting, we generalize the Steiner's formula obtained in \cite{Balogh-Steiner}.  See \cref{thm:Steiner}.

   \begin{thm}[Steiner's tube formula]\label{thm:Steiner-intro}
   In the setting of \cref{thm:weyl-intro}, assume furthermore that $S$ is a two-sided hypersurface. Then the volume of the half-tube $r \mapsto \mu(T^+(r))$ is smooth (or real-analytic) on $[0,r_0)$.
 Moreover, $\mu(T^+(r))$ has the following Taylor expansion at $r=0$:
     \begin{equation} \label{eq:Steiner_formula-intro}
        \mu(T^+(r))=\sum_{k\geq 1 }  \frac{1}{k!}\left(\int_{S} \diverg^{k-1}(\nabla\delta_{\mathrm{sign}}) \de \sigma\right) r^k,
    \end{equation}
    where $\mathrm{d}\sigma$ denotes the \sr perimeter measure on the hypersurface $S$.
   \end{thm}

\subsection{Polynomial character of iterated divergences}

The coefficients in the Weyl's (or Steiner's) tube formulas of \cref{thm:weyl-intro,,thm:Steiner-intro} are written in terms of the iterated divergences of $\nabla \delta$. In \cref{sec:polynomials}, we investigate to what extent these coefficients are \emph{polynomial} functions in the derivatives of the distance $\delta$ from the submanifold. This is the sub-Riemannian analogue of the fact, instrumental in the proof of Weyl's Euclidean tube formula, that the coefficients of the tube formula can be expressed as polynomials of the second fundamental form of the submanifold. In this introduction we only present the following simplified result for left-invariant sub-Riemannian structures (see \cref{lem:iter_div_poly,cor:iter_div_poly_m,rmk:invariance}). 

\begin{thm}
\label{thm:div_poly-intro}
   Let $M$ be a Lie group equipped with a left-invariant \sr structure, and with a left-invariant measure $\mu$. Let $Y_1,\ldots,Y_n$ be a global left-invariant frame for $TM$. Then, for all $k,m \in \N$ with $m\geq 1$, there exists a polynomial function $P^k_m$ in $n+n^2+1$ variables and with real coefficients, homogeneous of degree $k$, such that for any $C^2$ non-characteristic submanifold $S\subset M$ of codimension $m\geq 1$, it holds
      \begin{equation}\label{eq:Pk-intro}
        \diverg^k_{\mu/\delta^{m-1}}(\nabla\delta)=P^k_m\left(\dots,Y_\alpha\delta,\dots,Y_{\alpha}Y_{\beta}\delta,\dots,\tfrac{m-1}{\delta}\right),\qquad \text{on }  U\setminus S,
    \end{equation}
    where in the variables $Y_\alpha\delta$, $Y_\alpha Y_\beta\delta$ the indices $\alpha,\beta$ run over the set $\{1,\dots,n\}$.
    
    For the case of hypersurfaces ($m=1$) then the polynomials $P^k_1$ do not depend on the last variable.  Furthermore, if $S$ is also two-sided, \eqref{eq:Pk-intro} is valid on the whole $U$, replacing $\delta$ with $\mathrm{\delta}_{\mathrm{sign}}$.
   \end{thm}

\cref{thm:div_poly-intro} includes all Carnot groups. In such a setting one can hope to find a finite set of invariant generators for the algebra of the polynomials $P^k_m$. We do this for surfaces in the three-dimensional Heisenberg group in \cref{prop:forsedatogliere}, that we report here.

\begin{prop}
\label{prop:forsedatogliere-intro}
Let $\mathbb{H}$ be the three-dimensional Heisenberg group, equipped with a left-invariant measure $\mu$. Then, there exists polynomials $Q^k_{m}$ with real coefficients and in $5$ variables, such that for any $C^2$ non-characteristic submanifold $S\subset M$ with codimension $m \in \{1,2\}$ it holds
\begin{equation}
\label{eq:lasola-intro}
\diverg^{k}_{\mu/\delta^{m-1}}(\nabla\delta)= Q^k_m\left(F_1,F_2,F_3,F_4,\tfrac{m-1}{\delta}\right),\qquad \text{on } U\setminus S,
\end{equation}
where, for any given left-invariant and oriented orthonormal frame $\{X_1,X_2\}$, and letting $X_0$ be the Reeb vector field, we define the functions $F_i : U\setminus S \to \R$ by
    \begin{align}
        F_1 & :=X_1X_1\delta+X_2X_2\delta,  & F_2 & :=-(X_2\delta)(X_1 X_0\delta)+ (X_1\delta)(X_2 X_0\delta),\\ 
        F_3 & :=X_0\delta,  & F_4 & :=X_0 X_0\delta.
    \end{align}
   (The functions $F_1,\dots,F_4$ do not depend on the choice of the frame.) 
 Furthermore, if $S$ is a two-sided surface, \eqref{eq:lasola-intro} for $m=1$ holds on the whole $U$, replacing $\delta$ with $\delta_{\mathrm{sign}}$.
\end{prop}
\begin{rmk}[Independence on the frame]
The functions $F_1,\dots,F_4$ do not depend on the choice of the frame: they can be written in terms of the sub-Laplacian $\Delta$, the Reeb field $X_0$, the symplectic structure $J$, and the horizontal gradient $\nabla$:
\begin{equation}
F_1  = \Delta\delta, \qquad F_2  =  g(\nabla\delta, J\nabla X_0\delta), \qquad F_3  = X_0\delta, \qquad F_4  = X_0 X_0 \delta.
\end{equation}
\end{rmk}

\begin{rmk}
For two-sided hypersurfaces, an immediate consequence of \cref{prop:forsedatogliere-intro} paired with \cref{thm:Steiner-intro} is the following formula for the volume of the corresponding half-tube:
\begin{equation}
        \mu(T^+(r))=\sum_{k\geq 1 }  \frac{1}{k!}\left(\int_{S} Q^{k}_1( F_1,F_2,F_3,F_4)\de \sigma\right) r^k,
\end{equation}
where $F_1,\dots,F_4$ are defined as above replacing $\delta$ with $\delta_{\mathrm{sign}}$. Compare with \cite{Balogh-Steiner}, where the authors expressed the integrands in terms of $5$ basic functional invariants: $F_1,\dots,F_4$ and, in addition, $F_5:= (X_0X_1\delta)^2+(X_0X_2\delta)^2$. However, by using the Eikonal equation and its derivative, one can show that $F_5 = F_2^2$. We stress that, in order to prove \cref{prop:forsedatogliere-intro}, we do not compute explicitly the iterated divergences.
\end{rmk}

The extension of \cref{prop:forsedatogliere-intro} to higher-dimensional Heisenberg groups seems difficult. This is due to the fact that the corresponding group of isometries is relatively smaller in higher dimensions. See the discussion at the end of \cref{sec:polynomials}.

\subsection{Weyl's invariance for tubes around curves of the Heisenberg groups}

One of the deep results in \cite{Weyl-Tubes} is that the coefficients appearing in the tube formula have intrinsic nature. This can be stated as follows: any isometric embedding of a given Riemannian manifold $S$ in $\R^n$ has the same tube volume for small enough radii. Broadly speaking, we call this result \emph{Weyl's invariance}. To our best knowledge, \emph{all} proofs of this fact pass through the computation of the Jacobian determinant of the normal exponential map, explicit spherical integration, and Gauss' formula (relating the second fundamental form of an Euclidean submanifold with its Riemannian curvature), in order to cast the coefficients in terms of the Riemannian curvature of $S$. See \cite{Gray-Tubes}. We remark the following fact about Weyl's invariance result:
\begin{itemize}
\item[-] with analogous \emph{computational proof}, it holds on all space forms (even though the volume of tubes is no longer a polynomial);
\item[-] it does not extend to a general ambient Riemannian manifold $M$. For instance, the volume of a tube around a point depends on where the point is in $M$.
\end{itemize}

In the \sr setting, it is not clear how an invariance-type result can be stated. Think, for example, at a never-horizontal curve in the three-dimensional Heisenberg group, $\gamma: (0,1) \to \mathbb{H}$. The induced inner metric on $\gamma$ is infinite between any pair of distinct points, so that from the intrinsic metric viewpoint $\gamma$ is just a disjoint uncountable union of points. Any pair $\gamma,\gamma'$ of such curves will be diffeomorphic and ``isometric'', but one can easily find examples with different Weyl's tube formula. Extending the analogy to non-characteristic surfaces, the induced structure is that of a regular foliation with one-dimensional leaves, that is the disjoint union of an uncountable number of flat one-dimensional submanifolds. In any case, intrinsic ``isometries'' of these submanifolds seem to have little relevance with the tube formula.

We obtain a Weyl's invariance result for curves in the $(2d+1)$-dimensional Heisenberg group $\mathbb{H}_{2d+1}$. It establishes that the volume of small sub-Riemannian tubes around a non-characteristic curve depends only on the Riemannian length of the curve and the so-called Reeb angle.

\begin{defn}[Reeb angle]
\label{def:reebangle}
Let $S \subset \mathbb{H}_{2d+1}$ be an embedded $C^2$ non-characteristic submanifold of codimension $m\geq 1$. Its \emph{Reeb angle} $\theta_S: S\to \R$ is:
\begin{equation}
\theta_S(q):= \sup_{W\in T_q S} \frac{g_R(W, X_0)}{\|W\|_{g_R}},
\end{equation}
where $g_R$ is the canonical Riemannian extension of the \sr metric on $\mathbb H_{2d+1}$  and $X_0$ is the Reeb vector field.
\end{defn}

We can now state the invariance result. The corresponding statement for curves in the Euclidean space is due to Hotelling \cite{Hotelling}, which motivated \cite{Weyl-Tubes}. We remark that this is the first time that such a result is obtained in sub-Riemannian geometry. (cf.\ \cref{thm:invariance}).

\begin{thm}[Sub-Riemannian Hotelling's theorem]\label{thm:invariance-intro}
Let $\gamma,\gamma':[0,L]\to \mathbb{H}_{2d+1}$, be non-characteristic $C^2$ curves, parametrized with unit Riemannian speed. Denote with $\Gamma,\Gamma' \subset \mathbb{H}_{2d+1}$ the corresponding embedded submanifold. Assume that $\theta_{\Gamma}(\gamma_t)=\theta_{\Gamma'}(\gamma'_t)$ for all $t\in [0,L]$. Then, there exists $\epsilon>0$ such that
\begin{equation}
\mu(T_{\Gamma}(r)) = \mu(T_{\Gamma'}(r)), \qquad \forall \, r \in[0,\epsilon),
\end{equation}
where $\mu$ denotes the Lebesgue measure of $\mathbb{H}_{2d+1}$.
\end{thm}
\begin{rmk}
We were not able to extend the same strategy of proof of \cref{thm:invariance-intro} to general codimension. We believe that more quantities akin to the concept of Reeb angle will affect the volume of tubes for general embedded submanifolds.
\end{rmk}

\subsection{Structure of the paper} In \cref{sec:preliminaries} we introduce the necessary preliminaries  in sub-Riemannian geometry. In \cref{sec:regularity} we study the regularity of the distance from submanifolds, proving \cref{thm:tubolarneigh-intro}. In \cref{sec:weyl} we prove Weyl's and Steiner's expansions for the volume of tubes, namely \cref{thm:weyl-intro,thm:Steiner-intro}. In \cref{sec:polynomials} we discuss the polynomial character of the iterated divergences, and the proof of \cref{thm:div_poly-intro,prop:forsedatogliere-intro}. Finally, in \cref{sec:invariance}, we prove our Weyl's invariance result for curves in the Heisenberg group, namely \cref{thm:invariance-intro}.

\subsection{Acknowledgements}

This project has received funding from (i) the European Research Council (ERC) under the European Union's Horizon 2020 research and innovation programme (grant agreement GEOSUB, No. 945655); (ii) the PRIN project ``Optimal transport: new challenges across analysis and geometry'' funded by the Italian Ministry of University and Research; (iii) the ANR-DFG project ``CoRoMo'' (ANR-22-CE92-0077-01). The authors also wish to thank Antonio Lerario for helpful discussions around Weyl's tube formula.

\section{Preliminaries}\label{sec:preliminaries}

We recall some basic notions of \sr geometry, cf.\ \cite{ABB-srgeom}.

\subsection{Sub-Riemannian structures}
Let $M$ be a smooth, connected $n$-dimensional manifold. A smooth \sr structure on $M$ is defined by a \emph{generating family} of $N\in\N$ global smooth vector fields $\{X_1,\ldots,X_N\}$. The latter defines a (possibly rank-varying) \emph{distribution} $\D$:
\begin{equation}
\label{eq:gen_frame}
\D_p:=\text{span}\{X_1(p),\ldots,X_N(p)\}\subseteq T_pM,\qquad\forall\,p\in M.
\end{equation}
We assume that the generating family is \emph{bracket-generating}, i.e.\ the Lie algebra of smooth vector fields generated by $\{X_1,\dots,X_N\}$, evaluated at the point $p$, is equal to $T_p M$, for all $p\in M$. The generating family induces a norm on the distribution at $p$: 
\begin{equation}
\label{eq:induced_norm}
\|v\|^2_p:=\inf\left\{\sum_{i=1}^Nu_i^2 \,\Big|\, \sum_{i=1}^Nu_iX_i(p)=v\right\},\qquad\forall\,v\in\D_p,
\end{equation}
which, in turn, defines an inner product $g_p$ on $\D_p$ by polarization. The manifold $M$, equipped with the above structure, is said to be a \emph{smooth \sr manifold}. We say that $M$ is a \emph{real-analytic \sr manifold} if $M$ is a real-analytic manifold and the vector fields of the generating family are real-analytic.

A curve $\gamma : [0,1] \to M$ is \emph{horizontal} if it is absolutely continuous and there exists a \emph{control} $u\in L^\infty([0,1],\R^N)$ such that
\begin{equation}\label{eq:control}
\dot\gamma_t=\sum_{i=1}^N u_i(t) X_i(\gamma_t), \qquad \text{for a.e.}\, t \in [0,1].
\end{equation}
The \emph{length} of a horizontal curve is defined as:
\begin{equation}
\ell(\gamma) := \int_0^1 \|\dot\gamma_t\|_{\gamma_t} \de t.
\end{equation}
Finally, the \emph{\sr distance} between any two points $p,q\in
 M$ is defined as
\begin{equation}\label{eq:infimo}
\di(p,q) := \inf\{\ell(\gamma) \mid  \gamma \text{ horizontal curve joining $p$ and $q$} \}.
\end{equation}
By the Chow-Rashevskii Theorem, the bracket-generating assumption ensures that the distance $\di$ is finite, continuous and it induces the same topology as the manifold one.

\subsection{Horizontal gradient}

For an open set $U\subset M$ and $C^1$ function $a: U \to \R$, its \emph{horizontal gradient} is the horizontal vector field $\nabla a$ such that
\begin{equation}
g(\nabla a,v) = da(v),\qquad \forall\, v \in \D_p,\,\, \forall\,p\in U.
\end{equation}
In terms of the generating family $\{X_1,\dots,X_N\}$ it holds
\begin{equation}
\nabla a = \sum_{i=1}^N (X_i a) X_i \qquad \text{and} \qquad \|\nabla a\| = \sqrt{\sum_{i=1}^N (X_i a)^2}.
\end{equation}
(See \cite[Appendix A]{MR4623954} for a proof in the rank-varying case.)

%\begin{rmk}
%The above definition includes all classical constant-rank sub-Riemannian structures as in \cite{montgomerybook,Riffordbook}
% (where $\D\subset TM$ is a sub-bundle), but also general rank-varying sub-Riemannian structures.
%\end{rmk}

\subsection{Geodesics and Hamiltonian flow}
\label{sec:geod}

A \emph{geodesic} is a horizontal curve $\gamma :[0,1] \to M$, parametrized with constant speed, that is locally length-minimizing. The \emph{sub-Riemannian Hamiltonian} is the function $H : T^*M \to \R$, given by
\begin{equation}
\label{eq:Hamiltonian}
H(\lambda) := \frac{1}{2}\sum_{i=1}^N \langle \lambda, X_i \rangle^2, \qquad \lambda \in T^*M,
\end{equation}
where $\{X_1,\ldots,X_N\}$ is a generating family for the sub-Riemannian structure, and $\langle \lambda, \cdot \rangle $ denotes the action of the covector $\lambda$ on vectors. The \emph{Hamiltonian vector field} $\vec H$ on $T^*M$ is then defined by $\varsigma(\cdot,\vec H)=dH$, where $\varsigma\in\Lambda^2(T^*M)$ is the canonical symplectic form ($\varsigma=d\tau$, where $\tau \in \Lambda^1(T^*M)$ is the tautological form).

Solutions $\lambda : [0,1] \to T^*M$ of the \emph{Hamilton equations}
\begin{equation}\label{eq:Hamilton_eqs}
\dot{\lambda}_t = \vec{H}(\lambda_t),
\end{equation}
are called \emph{normal extremals} (which we always assume to be defined for all times, that is the case e.g.\ when $(M,\di)$ is complete). For $t\in \R$, we denote by $e^{t\vec H}:T^*M\to T^*M$ the \sr Hamiltonian flow at time $
t$. The projection of normal extremals $\gamma_t = \pi(\lambda_t)$ on $M$, where $\pi:T^*M\to M$ is the bundle projection, are geodesics, and are called \emph{normal geodesics}. If $\gamma$ is a normal geodesic with normal extremal $\lambda$, then  \eqref{eq:control} holds with controls $u_i(t) = \langle\lambda_t,X_i\rangle$, and its speed is $\| \dot\gamma \| = \sqrt{2H(\lambda)}$. In particular
\begin{equation}
\label{eq:speed}
\ell(\gamma|_{[0,t]}) = t \sqrt{2H(\lambda_0)},\qquad \forall\, t\in[0,1].
\end{equation} 
There is another class of locally length-minimizing curves, called \emph{abnormal geodesics}. To these curves $\gamma:[0,1]\to M$ correspond extremal lifts $\lambda: [0,1]\to T^*M$, which may not follow the Hamiltonian dynamics \eqref{eq:Hamilton_eqs}. 
Here, we only observe that for abnormal lifts it holds
\begin{equation}
\label{eq:abn}
\langle \lambda_t,\D_{\gamma_t}\rangle=0\quad \text{and} \quad \lambda_t\neq 0,\qquad \forall\, t\in[0,1] ,
\end{equation}
that is $H(\lambda_t)\equiv 0$. Note that a geodesic may be normal and abnormal at the same time. 

\subsection{Geodesics from a submanifold}

%An embedded submanifold $S\hookrightarrow M$ is \emph{properly embedded} (i.e.\ the embedding is a proper map) if and only if $S$ is closed as a subset of $M$. (To avoid confusion, we prefer to not use the term ``closed submanifold'' since $S$ may have boundary or be non-compact.) 

Consider a $C^1$ embedded submanifold $S \subset M$ (without boundary) of codimension $m\geq 1$. We define the \emph{distance from} $S$ as 
\begin{equation}
\label{eq:def_distance_S}
\ds: M\to [0,\infty),\qquad\ds(p):=\inf\left\{\di(p,q) \mid q\in  S\right\}.
\end{equation}
We say that a horizontal curve $\gamma:[0,1]\to M$ is a \emph{minimizing geodesic from} $S$ if it is a constant-speed length-minimizing curve, such that 
\begin{equation}
    \gamma_0\in S,\qquad \gamma_1 \in M\setminus S\qquad\text{and}\qquad \ell(\gamma)=\ds(\gamma_1).
\end{equation} 
If $\gamma:[0,1]\to M$ is a minimizing geodesic from $S$ there exists a corresponding extremal lift, $\lambda :[0,1]\to T^*M$, must satisfy the transversality conditions of \cite[Thm.\ 12.4]{AS-GeometricControl}, namely
\begin{equation}\label{eq:trcondition}
\langle \lambda_0, T_{\gamma_0} S \rangle=0.
\end{equation}
In other words, the initial covector $\lambda_0$ must belong to the \emph{annihilator bundle} 
\begin{equation}
AS := \{\lambda \in T^*M \mid \langle \lambda, T_{\pi(\lambda)} S\rangle = 0\} \subset T^*M.
\end{equation}
%If one allows $S$ to have a boundary, then minimizing geodesics from $S$ may start from a point on the boundary of $S$. In this case, extremal lifts may not satisfy condition \eqref{eq:trcondition}.

\begin{defn}[Non-characteristic submanifold]
    Let $S\subset M$ be a $C^1$ embedded submanifold of codimension $m\geq 1$. A point $q\in S$ is said to be \emph{non-characteristic} if
    \begin{equation}\label{eq:non-char}
        \D_q + T_qS =T_qM.
    \end{equation}
We say that $S$ is a \emph{non-characteristic submanifold} if it has no characteristic points.
\end{defn}

If $S$ is a non-characteristic submanifold without boundary, the curves that realize the distance from $S$ admit a unique normal lift satisfying the transversality condition \eqref{eq:trcondition}.

\begin{prop}\label{prop:no_abn}
 Let $M$ be a sub-Riemannian manifold and let $S \subset M$ be an embedded submanifold of codimension $m\geq 1$ and of class $C^1$. Let $\gamma : [0,1]\rightarrow M$ be a minimizing geodesic from $S$. Assume that $\gamma_0$ is non-characteristic and it is not a boundary point of $S$. Then, $\gamma$ is the projection of a unique normal extremal lift with initial covector $\lambda \in AS$.
\end{prop}
\begin{proof}
Let $\lambda:[0,1]\to M$ be an extremal lift of $\gamma$ satisfying the transversality condition \eqref{eq:trcondition}. The abnormal condition \eqref{eq:abn} at $t=0$ contradicts  \eqref{eq:trcondition}, so that no such lift can be abnormal. It follows that the lift is normal, namely there is $\lambda \in AS$ such that $\gamma_t = \pi\circ e^{t\vec{H}}(\lambda)$. If there were two such normal lifts, their difference would be an abnormal lift satisfying \eqref{eq:trcondition}, which as explained previously is a contradiction.
\end{proof}
\begin{rmk}
The curve $\gamma$ is also a minimizing geodesic for $\di$, and as such it \emph{can} admit abnormal lifts. As just proved, these will not satisfy the transversality condition.
\end{rmk}

We define the \emph{normal exponential map} as the restriction of the Hamiltonian flow to the annihilator bundle, namely
\begin{equation}
\label{eq:normal_exponential_map}
    E:\ann S\to M,\qquad E(\lambda):= \pi \circ e^{\vec{H}}(\lambda).
\end{equation}
We also define the \emph{unit annihilator bundle} $A^1 S \subset AS$ as the sphere bundle
\begin{equation}
    A^1S := \{\lambda \in AS \mid 2H(\lambda) = 1\}.
\end{equation}
Finally, for $r>0$ we let $E_r : A^1 S \to M$ be the \emph{normal exponential map at radius $r$}:
\begin{equation}
    E_r :A^1 S\to M, \qquad E_r(\lambda) := E(r\lambda) = \pi\circ e^{r\vec{H}}
(\lambda).
\end{equation}

\section{Regularity of the distance and tubular neighbourhoods}\label{sec:regularity}

In this section, we prove \cref{thm:tubolarneigh-intro}. We recall here its statement. 

\begin{thm}
\label{thm:tubolarneigh}
    Let $M$ be a sub-Riemannian manifold (smooth, without boundary, complete) and let $S \subset M$ be a non-characteristic submanifold of codimension $m\geq 1$ and of class $C^k$, with $k\geq 2$ (without boundary). Then, there exists a continuous function $\varepsilon : S \to \R_{>0}$, such that, letting
    \begin{equation}\label{eq:defV}
V:=\left\{\lambda \in AS \,\Big|\, \sqrt{2H(\lambda)} < \varepsilon(\pi(\lambda))\right\},
    \end{equation}
    the following statements hold:
    \begin{enumerate}[(i)]
        \item \label{item:tubolarneigh1} The restriction of the normal exponential map $E:V\to U:=E(V)$ is a $C^{k-1}$ diffeomorphism;
        \item \label{item:tubolarneigh2} For all $p=E(\lambda) \in U$ there exists a unique minimizing geodesic $\gamma:[0,1]\to M$ from $S$, which is normal, and it is given by $\gamma_t = E(t\lambda)$. In particular, on $V$ it holds
        \begin{equation}
                \ds \circ E = \sqrt{2H};
        \end{equation}
        \item \label{item:tubolarneigh3}  $\delta \in C^{k}(U\setminus S)$, with
                \begin{equation}
        \label{eq:eikonal}
            \|\nabla \ds\|=1,\qquad\text{on } U\setminus S;
        \end{equation}
\item \label{item:tubolarneigh4} $\delta^2\in C^{k}(U)$;
\item \label{item:tubolarneigh5} Let $X,Y$ be smooth (or real-analytic, if $M$ is real-analytic) vector fields. Then the functions $\delta$, $X\delta$, $YX\delta$ are smooth (or real-analytic) along any minimizing geodesic from $S$ contained in $U\setminus S$.
 \end{enumerate}
\end{thm}

The proof of \cref{thm:tubolarneigh} is split into two parts. In the first part, we prove \cref{item:tubolarneigh1,item:tubolarneigh2}. This part is quite standard, and it extends the proofs contained in \cite{Rossi22,MR4117982} to the case of a non-smooth, non-compact $S$. We present it here in order mainly to provide a self-consistent argument. From this, \cref{item:tubolarneigh3,item:tubolarneigh4} with $C^{k-1}$ regularity in place of $C^k$ follow. This is the starting point for the second part of the proof, where some new ideas are needed. There, we fully prove $C^k$ regularity in \cref{item:tubolarneigh3,item:tubolarneigh4}, and we prove \cref{item:tubolarneigh5}.
\begin{proof}[First part of the proof of \cref{thm:tubolarneigh}]
Note that $AS$ is a $C^{k-1}$ vector bundle, and $E: AS \to M$ is $C^{k-1}$. The non-characteristic assumptions is equivalent to the fact that $E$ has full rank at any point of the zero section of $AS$. By the inverse function theorem, it follows that for any point $q \in S$ there exists an open neighbourhood $V(q)\subset AS$ such that $E$ is a $C^{k-1}$ diffeomorphism when restricted to $V(q)$. Since $S$ is embedded and $2H$, restricted to the fibers of $AS$, is a well-defined norm, we can take
\begin{equation}
V(q) = V_{\varrho}(q) :=\left\{\lambda \in AS\mid \di(q,\pi(\lambda)) < \varrho, \,\, \sqrt{2H(\lambda)}<\varrho\right\},
\end{equation}
for some $\varrho>0$. Let then $\eta: S \to \R$ be
\begin{equation}
\eta(q):=\sup\{\varrho>0 \mid E:V_{\varrho}(q) \to E(V_{\varrho}(q)) \text{ is a diffeomorphism}\}>0.
\end{equation}
It is easy to prove that $\eta$ is $1$-Lipschitz w.r.t.\ $\di$. Then, we let $\varepsilon : S\to \R_{>0}$ be the function
\begin{equation}\label{eq:defeps}
\varepsilon(q):=\frac{1}{2}\min\Big\{\eta(q), \di(q,\mathrm{cl}(S)\setminus S)\Big\}>0.
\end{equation}
Note that $\di(q,\mathrm{cl}(S)\setminus S)>0$, for every $q\in S$, since we have taken $S$ to be a manifold without boundary. As the minimum of two $1$-Lipschitz functions, $\varepsilon$ is Lipschitz, and hence continuous. (Of course if $q\to q_0 \in\mathrm{cl}(S)\setminus S$ then $\varepsilon \to 0$.)

We now show that the set $V$ defined as in \eqref{eq:defV} satisfies \cref{item:tubolarneigh1,item:tubolarneigh2}. 

\textbf{Proof of \cref{item:tubolarneigh1}.} Let $\lambda_1,\lambda_2 \in V$, with $q_i=\pi(\lambda_i)$, and assume that $E(\lambda_1)=E(\lambda_2)=p$. The curves $\gamma_i:[0,1]\to M$ given by $\gamma_{i,t}=E(t\lambda_i)$ are horizontal, with length $\sqrt{2H(\lambda_i)}$, joining $q_i$ with $p$. Without loss of generality assume that $\varepsilon(q_1)\leq \varepsilon(q_2)$. Then $\di(q_1,q_2)<  \varepsilon(q_1) +\varepsilon(q_2)   \leq 2\varepsilon(q_2)\leq \eta(q_2)$. Thus by construction $\lambda_1 \in V_{\varepsilon(q_2)}(q_2)\subset V_{\eta(q_2)}(q_2)$. On this set $E$ is a diffeomorphism and thus $\lambda_1=\lambda_2$. It follows that $E: V\to E(V)$ is injective, and thus a $C^{k-1}$ diffeomorphism, proving \cref{item:tubolarneigh1}.

\textbf{Proof of \cref{item:tubolarneigh2}.} Let $p = E(\lambda)$ with $\lambda \in V$. The curve $\gamma :[0,1]\to M$ given by $\gamma_t = E(t\lambda)$ is horizontal, with length $\sqrt{2H(\lambda)}$, and joins $q=\pi(\lambda)\in S$ with $p$. It follows that
\begin{equation}\label{eq:pairing}
\delta(E(\lambda)) \leq \sqrt{2H(\lambda)} =\ell(\gamma).
\end{equation}
%To prove equality, some care is needed since $S$ may have boundary, and this is where $\tilde{S}$ plays a role. 
%Let $\tilde{\delta}$ be the distance from $\tilde{S}$. Indeed it holds $\tilde{\delta}\leq \delta$ (globally). 
Let $(\gamma_n)_n$, $\gamma_n:[0,1]\to M$ be a sequence of horizontal curves with $\gamma_n(0) \in S$, $\gamma_n(1) = p$, such that $\ell(\gamma_n) \to \delta(p)$. Up to extraction (and since $(M,\di)$ is complete), $\gamma_n$ converges uniformly to a curve $\bar{\gamma}$ from $\mathrm{cl}(S)$ to $p$ such that $\ell(\bar{\gamma}) = \delta(p)$. We claim that the initial point $\bar{q}=\bar{\gamma}(0)$ cannot be in $\mathrm{cl}(S)\setminus S$. In fact since $q\in S$ then $\di(q,\mathrm{cl}(S)\setminus S)>0$ and on the other hand
\begin{align}
\di(q,\bar{q}) & \leq \di(q,p) + \di(p,\bar{q}) \leq \ell(\gamma) + \delta(p) \leq 2\ell(\gamma) < 2\varepsilon(q) \leq \di(q,\mathrm{cl}(S)\setminus S).
\end{align}
Therefore $\bar{\gamma}$ is a minimizing geodesic from $S$, starting from $\bar{q} \in S$. Furthermore, $S$ is non-characteristic (and without boundary). By \cref{prop:no_abn}, $\bar{\gamma}$ is the projection of a (unique) normal lift: there is $\bar{\lambda}\in AS$ with $\pi(\bar{\lambda}) = \bar{q}$ such that $\bar{\gamma}(t) = E(t\bar{\lambda})$. Moreover
\begin{equation}
\sqrt{2H(\bar{\lambda})} =\ell(\bar{\gamma}) =\delta(p) \leq \ell(\gamma) =\sqrt{2H(\lambda)} < \varepsilon(q) <\eta(q),
\end{equation}
and also $d(q,\bar{q}) <2\varepsilon(q)\leq \eta(q)$. It follows that $\bar{\lambda} \in V_{\eta(q)}(q)$ so that $\bar{\lambda}=\lambda$, and thus $\gamma = \bar{\gamma}$.

Since $(\gamma_n)_n$ was arbitrary, it follows that $\gamma:[0,1]\to M$ is the unique minimizing geodesic from $S$ to $p = E(\lambda)$, given by $\gamma_t=E(t\lambda)$, which is normal. This concludes the proof of \cref{item:tubolarneigh2}.

\textbf{Proof of \cref{item:tubolarneigh3,item:tubolarneigh4} for $C^{k-1}$.} By the above items, on $U =E(V)$, it holds
\begin{equation}
\delta^2 = 2H \circ E^{-1}.
\end{equation}
Since $E:V \to U$ is a $C^{k-1}$ diffeomorphism and $2H:T^*M\to \R$ is smooth, we obtain that $\delta^2 \in C^{k-1}(U)$ and $\delta \in C^{k-1}(U\setminus S)$. 

The fact that the Eikonal equation \eqref{eq:eikonal} holds is classical.  Indeed, since $\delta$ is $1$-Lipschitz w.r.t.\ $\di$, then (see e.g.\ \cite[Thm.\ 8]{FHK-Sobolev}), it holds
\begin{equation}\label{eq:leq}
\|\nabla\delta \|=\sqrt{\sum_{i=1}^N (X_i \delta)^2} \leq 1,\qquad \text{almost everywhere on $M$}.
\end{equation}
In particular \eqref{eq:leq} holds everywhere on $U\setminus S$. To prove the opposite inequality, let $\gamma_t = E(t\lambda)$ for $\lambda \in V$. The curve $\gamma:[0,1]\to M$ has speed $\|\dot\gamma_t\|=\sqrt{2H(\lambda)}$, and \cref{item:tubolarneigh2} implies
\begin{equation}
\delta(\gamma_t) = \sqrt{2H(t\lambda)} = t\|\dot{\gamma}\|, \qquad\, \forall \, t\in [0,1].
\end{equation}
Differentiating w.r.t.\ $t$ we get $g(\nabla\delta,\dot\gamma) =  \|\dot\gamma\|$. By the Cauchy-Schwartz inequality $\|\nabla\delta\|\geq 1$.
\end{proof}

To prepare for the second part of the proof of \cref{thm:tubolarneigh}, we need some preliminary lemmas. Denote with $\Phi: \R\times T^*M\to T^*M$ the \emph{extended Hamiltonian flow}, namely 
\begin{equation}
\label{eq:extended_hflow}
    \Phi(t,\lambda):=e^{t\vec H}(\lambda),\qquad\forall\,(t,\lambda)\in\R\times T^*M.
\end{equation}
Furthermore, denote with $\Psi: \R\times T(T^*M)\to T(T^*M)$ the linearisation of $\Phi$, namely
\begin{equation}\label{eq:lin_ham_flow}
\Psi(t,\xi):=d_{\pi(\xi)}\Phi(t,\cdot)(\xi), \qquad \forall\,(t,\xi) \in \R \times T(T^*M).
\end{equation}

The next result follows from standard regularity theory for ODEs.
\begin{lem}
\label{lem:ham_flow_analytic}
    Let $M$ be a smooth (or real-analytic) \sr ma\-nifold. Then $\Phi$ and $\Psi$ are smooth (or real-analytic). 
   \end{lem}

We also recall \cite[Lemma 2.15]{Riffordbook}. There, the \sr structure is assumed to have constant rank, but the statement holds, with the same proof, in our setting.
\begin{lem}\label{lem:rifford}
Let $p\neq q \in M$ be such that there exists a function $\phi : M \to \R$, differentiable at $p$, such that
\begin{equation}
\phi(p) = \di^2(p,q),\qquad\text{and}\qquad  \di^2(q,z)\geq \phi(z),\quad \forall\, z\in M.
\end{equation}
Then there is a unique minimizing geodesic $\gamma:[0,1]\to M$ between $p$ and $q$. It is the projection of a unique normal extremal $\lambda : [0,1]\to T^*M$, that satisfies $\lambda_1  = \tfrac{1}{2}d_p\phi$.
\end{lem}

The following lemma relates the differential of $\delta$ at different points along geodesics from $S$.

\begin{lem}
\label{lem:diff_delta}
Let $V\subset AS$, $U\subset M$ be the neighbourhoods of \cref{thm:tubolarneigh}. Let $p\in U \setminus S$ and let $\gamma:[0,1]\to M$ be the unique minimizing geodesic from $S$ to $p = E(\lambda)$, with $\lambda \in V$. Then
  \begin{equation}
      \label{eq:identity_final_covector}
      \Phi(t,\lambda)=\ds(p)d_{\gamma_t}\ds,\qquad\forall\,t\in (0,1],
  \end{equation}
  and moreover
    \begin{equation}
    \label{eq:tautology_distance_t_s}
        d_{\gamma_t}\ds= \Phi ((t-s)\ds(p), d_{\gamma_s}\ds),\qquad \forall\,t,s\in (0,1].
    \end{equation} 
\end{lem}

\begin{proof}
    Let $\lambda\in V\subset AS$ be the covector such that $E(\lambda)=p$. The curve $\gamma$ is characterized as
    \begin{equation}
    \label{eq:normal_geod}
        \gamma_t=\pi\circ \Phi(t,\lambda)\qquad\forall\,t\in [0,1].
    \end{equation}
Let $q:=\pi(\lambda)$ and $\phi:=\ds^2$. Then $\phi \in C^1(U\setminus S)$ by \cref{thm:tubolarneigh} and, in addition, it satisfies
    \begin{equation}
\phi(p) = \di^2(p,q),\qquad\text{and}\qquad  \di^2(q,z)\geq \phi(z),\quad \forall\, z\in M.
\end{equation}
    It follows from \cref{lem:rifford} that $\gamma$ is the projection of a unique normal extremal with final covector $\frac{1}{2} d_p\phi=\ds(p)d_p\ds\in T_p^*M$.  Thus, by Cauchy-Lipschitz, we must have $\Phi(1,\lambda)=\frac{1}{2}d_p\phi$. Repeating the same argument for $E(t\lambda) \in U$, for every $t\in (0,1]$, we obtain that 
    \begin{equation}
    \label{eq:tautology_distance}
    \Phi(1,t\lambda)=
        \ds(\gamma_t)d_{\gamma_t}\ds=t\ds(p)d_{\gamma_t}\ds.
    \end{equation}    
  In addition, since $H$ is homogeneous of degree $2$, $\Phi(1,t\lambda)=t\Phi(t,\lambda)$, thus \eqref{eq:tautology_distance} implies \eqref{eq:identity_final_covector}.

From \eqref{eq:identity_final_covector}, evaluated at $t,s\in (0,1]$, and using the group property of $\Phi(\cdot,\lambda)$, we deduce
  \begin{equation}
  \label{eq:identity_intial_covector}
      \lambda=\ds(p)\Phi(-t\ds(p),d_{\gamma_t}\ds)=\ds(p)\Phi(-s\ds(p),d_{\gamma_s}\ds),
  \end{equation}
  from which \eqref{eq:tautology_distance_t_s} follows. 
\end{proof}

We can now complete the proof of  \cref{thm:tubolarneigh}. Remember that we have already proved \cref{item:tubolarneigh1,item:tubolarneigh2} and \cref{item:tubolarneigh3,item:tubolarneigh4} for $C^{k-1}$ regularity.

\begin{proof}[Second part of the proof of \cref{thm:tubolarneigh}]
Let $V\subset AS$ be the neighbourhood \eqref{eq:defV}.

\textbf{Proof of \cref{item:tubolarneigh3,item:tubolarneigh4}, conclusion.} In the first part of the proof we proved that $\ds \in C^{k-1}(U\setminus S)$ and $\ds^2 \in C^{k-1}(U)$. Using \eqref{eq:identity_final_covector}, at $t=1$, for all $p \in  U\setminus S$ it holds
    \begin{equation}
    \label{eq:dds}
	d_{p}\ds = \frac{1}{\ds(p)} \Phi(1,E^{-1}(p)) = \Phi \left( \ds(p),\frac{E^{-1}(p)}{\ds(p)}\right).
 %=\frac{\Phi \left( 1,E^{-1}(p)\right)}{\ds(p)}
    \end{equation}
It follows that $\ds\in C^k(U\setminus S)$. Similarly from \eqref{eq:identity_final_covector} we deduce 
    \begin{equation}
        d_p\ds^2=2\ds(p)d_p\ds=2\Phi(1,E^{-1}(p)). 
    \end{equation}
    This shows that $\ds^2\in C^k(U)$.
    
\textbf{Proof of \cref{item:tubolarneigh5}.} Let $\gamma:[0,1]\to M$ be the minimizing geodesic from $S$, with $\gamma_1 =p \in U \setminus S$. First of all, $\delta(\gamma_t) = t\delta(p)$, which is an analytic function of $t$. Second of all, as a consequence of \cref{lem:diff_delta}, the map $(0,1]\ni t\mapsto d_{\gamma_t}\delta$ is smooth (or real-analytic). It follows that $t\mapsto X\delta(\gamma_t) = (d_{\gamma_t}\delta)(X)$ is smooth (or real-analytic).

The proof for the regularity of $YX\delta$ is as follows. For a vector field $X$, let $h_X:T^*M\to\R$ defined as $h_X(\lambda)=\langle\lambda,X\rangle$.  Then, we write
        \begin{equation}
        \label{eq:YXds}
            YX\ds |_{p}=d_p(X\ds)Y= h_Y (d_p( h_X(d\ds))) = h_Y\circ d_{d_p\ds}h_X\circ d_p(d\ds),
        \end{equation}
    where  $d\delta : U\setminus S \to T^*M$ is of class $C^{k-1}$. Since $X,Y$ are smooth (or real-analytic) so are $h_Y$ and $dh_X$. We are left to show that $d(d\ds):T(U\setminus S)\to T(T^*M)$ is smooth (or real-analytic) along $\gamma$. Define $E^+:(0,1)\times A S\to M$ by setting $E^+(t,\lambda):=\pi\circ\Phi(t,\lambda)$, and observe that it is smooth (or real-analytic) w.r.t.\ $t$. Then, \eqref{eq:tautology_distance_t_s} with $s=1$ implies the following: for all $r>0$ and all $\lambda \in V\subset AS$ with $\sqrt{2H(\lambda)}=r$ it holds
            \begin{equation}
        \label{eq:tautology_distance_E1}
            d_{\unitexp(t,\lambda)}\ds= \Phi ((t-1)r, d_{\unitexp(1,\lambda)}\ds),\qquad \forall\,t \in (0,1].    
        \end{equation}
        Fix a basis $\{v_1,\ldots,v_{n-1}\} \in T_\lambda(\ann S\cap\{\sqrt{2H}=r\})$ and define $Y_i(t):=(d_\lambda\unitexp(t,\cdot))(v_i)$. Then, for $t\in (0,1]$, setting $Y_n(t):=\dot\gamma_t$, the family $\{Y_1(t),\ldots,Y_{n}(t)\}$ is a smooth (or real-analytic) moving frame for $T_{\gamma_t}M$ along $\gamma$. On the one hand, for every $i=1,\ldots,n-1$, differentiating \eqref{eq:tautology_distance_E1} w.r.t.\ $\lambda$, and evaluating along $\gamma_t$, we obtain
         \begin{equation}
         \label{eq:aux1_d2_delta}
            d^2_{\gamma_t}\delta(Y_i(t))=d_{\lambda}(d\delta\circ \unitexp(t,\cdot))(v_i)=\Psi((t-1)r, d^2_{p}\ds(Y_i(1))),\qquad \forall\,t\in (0,1],
         \end{equation}
where we used the shorthand $d^2_q \delta = d_q (d\delta)$, and where $\Psi$ is the linearisation of $\Phi$, see \eqref{eq:lin_ham_flow}. On the other hand, a direct computation shows that
         \begin{equation}
         \label{eq:aux2_d2_delta}
             d^2_{\gamma_t}\delta(Y_n(t))=\partial_t\Phi((t-1)r,d_{p}\delta) = r \vec{H}|_{\Phi((t-1)r,d_{p}\delta)}
                        ,\qquad \forall\,t \in (0,1].
         \end{equation}
         Therefore, we conclude that (the matrix representation of) the map $(0,1]\ni t\mapsto d_{\gamma_t}(d\delta)$ is smooth (or real-analytic) thanks to \cref{lem:ham_flow_analytic} and the analogous regularity of $d\delta$. 
\end{proof}

\subsection{The case of two-sided hypersurfaces}\label{sec:twosided}

Let $k\geq 2$. We say that a $C^k$ embedded submanifold $S\subset M$ (without boundary) is a two-sided hypersurface if has codimension one and it admits a never-vanishing transverse and continuous vector field $N$. (E.g.\ $S = \partial\Omega$ for a properly embedded manifold with boundary $\Omega$.) In this case, we have a splitting
\begin{equation}\label{eq:splitAS}
AS = S\sqcup A^+S \sqcup A^{-}S, \qquad A^{\pm} S = \{\lambda \in AS \mid \pm \langle \lambda,N\rangle >0\},
\end{equation}
where $S$ above is identified with the zero section. The neighbourhood $U\subset M$ of \cref{thm:tubolarneigh} splits according to \eqref{eq:splitAS} as
\begin{equation}
U = S \sqcup U^+ \sqcup U^-, \qquad U^{\pm} := E\left(A^{\pm} S \cap V\right).
\end{equation}
In this case, we define the \emph{signed distance from} $S$ as the function $\delta_{\mathrm{sign}}: U\to \R$
    \begin{equation} 
        \ds_{\mathrm{sign}}(p):=\begin{cases} \delta(p) &\text{if }p\in U^+,\\
            -\delta(p) &\text{if }p\in U^-,\\
            0 & \text{if } p \in S.
        \end{cases}
    \end{equation}

\begin{cor}\label{cor:twoside}
In the same setting of \cref{thm:tubolarneigh}, assuming furthermore that $S$ is a two-sided non-characteristic hypersurface, \cref{item:tubolarneigh3,item:tubolarneigh4,item:tubolarneigh5} hold on up to $S$ (i.e.\ on the whole $U$), replacing $\delta$ with $\delta_{\mathrm{sign}}$.
\end{cor}
\begin{proof}
We only have to focus on \cref{item:tubolarneigh3,item:tubolarneigh5}. For any $q\in S$ there is a neighbourhood $W$ and $f\in C^k(W)$ such that 
    \begin{equation}
       S\cap W = \{f=0\},\qquad \text{with} \qquad df|_{S\cap W}\neq 0.
    \end{equation}
	Since $AS\cap \pi^{-1}(W)$ has one-dimensional fiber, any $\lambda \in AS\cap \pi^{-1}(W)$ can be written as $\lambda = z df$, and the map $\lambda \mapsto (\pi(\lambda),z)$ defines a $C^{k-1}$ chart for $AS$. Therefore, arguing as in the first part of the proof of \cref{thm:tubolarneigh}, we have that
	\begin{equation}
	\delta_{\mathrm{sign}}(E(\lambda)) = z \sqrt{2H(df)}, \qquad \forall\, \lambda=z df \in V \cap \pi^{-1}(W).
	\end{equation}
	We conclude that $\delta_{\mathrm{sign}} \in C^{k-1}(U)$ and that \eqref{eq:eikonal} holds.
	
	Since $\delta_{\mathrm{sign}} \in C^{k-1}(U)$, taking the limit for $t\to 0$ (or $s\to 0$) in \eqref{eq:identity_final_covector} and \eqref{eq:tautology_distance_t_s} one obtains that \cref{lem:diff_delta} holds (for the signed distance) at all $t,s\in[0,1]$. Hence, applying \cref{lem:diff_delta} for $p\in U$, $t=1$ and $s=0$ and observing that $d\delta_{\mathrm{sign}}=df /  \sqrt{2H(df)}$ on $S\cap W$, we obtain
	\begin{equation}
	d_p\delta_{\mathrm{sign}} = \Phi\left(\delta_{\mathrm{sign}}(p), \left.\frac{d f}{\sqrt{2H(df)}}\right|_{\pi\circ E^{-1}(p)}\right).
	\end{equation}
	Since $E^{-1} :U \to V$ is a $C^{k-1}$ diffeomorphism we obtain that $\delta_{\mathrm{sign}} \in C^{k}(U)$. Once this is proved, the proof of \cref{item:tubolarneigh5} (replacing $\delta$ with $\delta_{\mathrm{sign}}$) also extends up to $t=0$.
 \end{proof}

\subsection{Uniform tubular neighbourhoods}

The function $\varepsilon: S\to \R_{>0}$ of \cref{thm:tubolarneigh} may tend to zero either when $S$ is non-compact or $\mathrm{cl}(S)\setminus S$ is non-empty. Under suitable extrinsic conditions on $S$, the tubular neighbourhood of \cref{thm:tubolarneigh} can be chosen to be uniform.

\begin{defn}\label{def:extendible}
An embedded submanifold $S  \subset M$ of class $C^k$ satisfies the \emph{extendibility property} if there exists an embedded submanifold $\tilde{S}  \subset M$ of class $C^k$ of the same dimension, and without boundary, such that 
\begin{equation}\label{eq:separation}
    \mathrm{cl}(S) \subset \tilde{S}.
\end{equation}
If $S$ is non-characteristic, we also ask $\tilde{S}$ to be non-characteristic.
\end{defn}

\begin{example}
Any closed manifold $S$ (i.e., compact without boundary) embedded in $M$ satisfies the extendibility property with $\tilde{S}=S$. If $\mathrm{cl}(S)\subset M$ is an embedded submanifold with boundary, then one can build $\tilde{S}$ such that \eqref{eq:separation} is satisfied, via a suitable flowout from its boundary \cite[p.\ 217]{MR2954043}; if $\mathrm{cl}(S)$ has no characteristic points, $\tilde{S}$ can be chosen to satisfy the same property. The \emph{localizing neighbourhoods} used in \cite{Balogh-Steiner,BB-Steiner3D} also satisfy \cref{def:extendible}. Note that $S$ may satisfy the extendibility property without $\mathrm{cl}(S)$ being an embedded manifold with boundary (e.g., any open set $S\subset \R^2 \hookrightarrow \R^3$, with irregular frontier).
An example without the extendibility property is the figure $\infty$ with the central point removed, embedded in $\R^2$.
\end{example}

\begin{thm}\label{thm:uniformneigh}
In the setting of \cref{thm:tubolarneigh}, assume furthermore that $S$ is bounded and satisfies the extendibility property. Then, there exists $r_0=r_0(S)>0$ such that \cref{item:tubolarneigh1,item:tubolarneigh2,item:tubolarneigh3,item:tubolarneigh4,item:tubolarneigh5} hold for
\begin{equation}\label{eq:Vuniform}
V:=\left\{ \lambda \in AS \,\Big|\, \sqrt{2H(\lambda)}< r_0 \right\}.
\end{equation}
\end{thm}

\begin{proof}
We apply \cref{thm:tubolarneigh} to $\tilde{S}$. Let $\tilde{\varepsilon}:\tilde{S}\to \R_{>0}$ be the corresponding function, and $\tilde{V}\subset A\tilde{S}$, $\tilde{U}=E(\tilde{V})\subset M$ be the corresponding neighbourhoods such that \cref{item:tubolarneigh1,item:tubolarneigh2,item:tubolarneigh3,item:tubolarneigh4,item:tubolarneigh5} hold for the distance from $\tilde{S}$, denoted by  $\tilde{\delta}$. Thanks to \eqref{eq:separation}, we have
\begin{equation}
r:=\min\left\{ \tilde{\varepsilon}(q)\mid q \in \mathrm{cl}(S)\right\} >0.
\end{equation}
Let $V_r:=\{ \lambda \in AS \mid \sqrt{2H(\lambda)}< r \}\subset AS\subseteq A\tilde{S}$. By construction, on $U_r:=E(V_r)\subseteq \tilde{U}$, it holds $\tilde{\delta} \equiv \delta$. (The inequality $\tilde{\delta}\leq \delta$ is obvious. But then, if $p\in U_r$, there is a unique minimizing geodesic from $\tilde{S}$ given by $\gamma_t=E(t\lambda)$ with $\lambda =E^{-1}(p)$ $\in V_r$. This geodesic starts from $S$ so that it is also minimizing from $S$.)
Then, \cref{item:tubolarneigh1,item:tubolarneigh2,item:tubolarneigh3,item:tubolarneigh4,item:tubolarneigh5} hold also for $\delta$ on $U_r = E(V_r)$. Finally, set $r_0 = r_0(S)$ to be the supremum of all such $r$. (This does not depend on the choice of $\tilde{S}$.)
\end{proof}
% \begin{rmk}\label{rmk:injrad}
% If there is $V$ of form \eqref{eq:Vuniform} such that \cref{item:tubolarneigh1,item:tubolarneigh2,item:tubolarneigh3,item:tubolarneigh4,item:tubolarneigh5} of \cref{thm:tubolarneigh} hold, we say that $S$ has \emph{positive injectivity radius}. Note that if $\mathrm{cl}(S)\supsetneq S$ in $M$, then $r_0$ can be smaller of the supremum of the $r>0$ s.t.\ $E$ is a diffeomorphism on $V_r=\{\lambda \in AS \mid \sqrt{2H(\lambda)}<r\}$. See \eqref{eq:defeps}.
% \end{rmk}

\section{The volume of tubes}\label{sec:weyl}

We now prove asymptotic formulas for the volume of sub-Riemannian tubes (cf.\ \cref{def:tube}). We begin with \cref{thm:weyl-intro}, of which we recall the statement for convenience.

\begin{thm}
\label{thm:weyl}
Let $M$ be a smooth (or real-analytic) \sr manifold, equipped with a smooth (or real-analytic) measure $\mu$. Let $S \subset M$ be a bounded non-characteristic embedded submanifold of codimension $m\geq 1$, of class $C^2$ (without boundary) and with the extendibility property. Let $r_0=r_0(S)>0$ be its injectivity radius. Then the volume of the tube $r\mapsto \mu(T(r))$ is smooth (or real-analytic) on 
    $[0,r_0)$. Furthermore, there exists a continuous density $\sigma_m$ on $A^1S$, defined by
    \begin{equation}\label{eq:sigmam}
        \sigma_m := \lim_{r \to 0} \frac{E^*_r(\iota_{\nabla\delta} \mu)}{r^{m-1}} ,
    \end{equation}
    and continuous functions $w_{m}^{(j)}:A^1S \to \R$ defined for $j\in \N$ by
    \begin{equation}\label{eq:wmj}
        w_{m}^{(j)} := \lim_{r\to 0} \diverg^{j}_{\mu/\delta^{m-1}}(\nabla\delta) \circ E_r,
    \end{equation}
    such that $\mu(T(r))$ has the following Taylor expansion at $r=0$:
    \begin{equation}
    \label{eq:expansion}
         \mu(T(r))=\sum_{\substack{k\geq m \\ k-m \text{ even}}}  \frac{1}{k(k-m)!}\left(\int_{A^1 S} w_m^{(k-m)} \,\de\sigma_m \right) r^k,
    \end{equation}
    where $\de\sigma_m$ denotes the measure induced by $\sigma_m$.
    \end{thm}

%Before proving the theorem, let us recall that the annihilator bundle $\pi_\ann: \ann S\to S$ of $S$ is a $C^1$ vector bundle over $S$ and, under the non-characteristic assumption, the restriction of the sub-Riemannian Hamiltonian to $\ann S$ induces a scalar product on its fibers. 

\begin{proof}
We apply, and use the notation of, \cref{thm:tubolarneigh,thm:uniformneigh}. Then, for $r\in (0,r_0)$, it holds
\begin{equation}
\mu(T(r)) = \int_{E\left(\left\{\lambda \in AS \mid \sqrt{2H(\lambda)}<r\right\}\right)} \mu=\int_{AS\cap \left\{\sqrt{2H}<r\right\}} E^*\mu.
\end{equation}
Let $\tilde{S}$ be the extension coming from the extendibility property, see \cref{def:extendible}. Thanks to the latter, we can cover $S$ with a \emph{finite} number of charts where $S$ is the graph of a $C^2$ function. Thus, by a partition of unity argument, it is sufficient to consider the case in which $\mathrm{cl}(S)$ is contained in a single chart and is diffeomorphic to the closure of an open set of $\R^{n-m}$. More precisely, we assume there are smooth (or real-analytic) coordinates $(x,y)\in  \R^{n-m}\times \R^m$ on an open neighbourhood $\mathcal{O}\subset M$ such that $\mathrm{cl}(S)\subset \mathcal{O}$, and there exists a $C^2$ function $h:\R^{n-m}\to \R^m$ and an open set $B\subset \R^{n-m}$ such that 
\begin{equation}
S = \{(x,h(x)) \in \R^{n-m} \times \R^m \mid x\in B\}.
\end{equation}
In these coordinates let $f : \R^n \to \R$ be  the smooth (or real-analytic) positive function such that
\begin{equation}
\mu|_{\mathcal{O}} = f d x\wedge d y.
\end{equation}
Choose $C^1$ sections $\nu_1,\dots,\nu_m$ of $\ann \tilde{S}$ (defined on a neighbourhood of $\mathrm{cl}(S)$), orthonormal w.r.t.\ the Hamiltonian scalar product, so that any $\lambda \in \pi^{-1}(\mathcal{O}\cap \mathrm{cl}(S))\cap A\tilde{S}\subseteq A\tilde{S}$ can be written as $\lambda = \sum_{i=1}^m p_i \nu_i$. This yields coordinates $(x,p) \in \mathrm{cl}(B)\times \R^{m}$ on $\pi^{-1}(\mathcal{O}\cap \mathrm{cl}(S))\cap A\tilde{S}$, such that $2H(x,p) = |p|^2$. In these coordinates, we have
\begin{equation}
E^*\mu = J(x,p) f(E(x,p)) d x\wedge\!d p, \qquad J(x,p) = \det\left(\frac{\partial E}{\partial x},\frac{\partial E}{\partial p} \right).
\end{equation}
Here, $E: \mathrm{cl}(B)\times \R^{m} \to \R^n$ denotes the normal exponential map in coordinates $(x,p)$, namely $E(x,p) = \pi \circ e^{\vec{H}}(\sum_{i=1}^m p_i \nu_i|_{(x,h(x))})$. Taking polar coordinates $p=\varrho u$ in the fibers, we have
\begin{equation}\label{eq:98}
\mu(T(r)) = \int_0^r \varrho^{m-1} \left(\int_{B \times \mathbb{S}^{m-1}} J(x,\varrho u) f(E(x,\varrho u)) \de x \de u \right) \!\de\varrho,
\end{equation}
where $\mathrm{d} u$ is the standard measure on the unit sphere $\mathbb{S}^{m-1}$.

Under our assumptions, $E(x,p)$ is $C^1$ and $J(x,p)$ is $C^0$, so that \eqref{eq:98} only suggests $C^1$ regularity in $r$. However, we claim that the map $\Theta : \mathrm{cl}(B)\times [0,\infty)\times \mathbb{S}^{m-1}\to \R$ given by
\begin{equation}\label{eq:Theta}
\Theta(x,\varrho,u) = J(x,\varrho u)f(E(x,\varrho u)),
\end{equation}
is smooth w.r.t.\ $\varrho$, and all its derivatives w.r.t.\ $\varrho$ are continuous. Furthermore, in the real-analytic case, $\Theta$ is real-analytic w.r.t.\ $\varrho$, locally uniformly w.r.t.\ $(x,u)$. By this we mean that, for any compact $K\subset \mathrm{cl}(B)\times [0,\infty)\times \mathbb{S}^{m-1}$, there exists $C>0$ depending only on $K$ such that
\begin{equation}
\sup_{(x,\varrho,u)\in K} \frac{\partial^k\Theta}{\partial\varrho^k}(x,\varrho,u) \leq C^{k+1}k!,\qquad \forall\, k\in \N.
\end{equation}
From the claim, it follows that $r\mapsto \mu(T(r))$ is smooth (or real-analytic) on $[0,r_0)$, and one can differentiate under the integral sign an infinite number of times in \eqref{eq:98}.

The proof of the claim is based on the following observation. Recall the extended Hamiltonian flow $\Phi : \R\times T^*M\to T^*M$ defined in \eqref{eq:extended_hflow}. The reparametrization identity $\Phi(1,t \lambda)=t\Phi(t,\lambda)$ for all $(t,\lambda)\in T^*M$ corresponds, in our coordinates, to the following:
\begin{equation}\label{eq:Et}
E(x,\varrho u) =  \pi\circ\Phi\left(\varrho,x,h(x), \sum_{i=1}^m  u_i \nu_i|_{(x,h(x))}\right) =:E^\varrho(x,u).
\end{equation}
Therefore, for every $(x,\varrho,u)\in \mathrm{cl}(B)\times [0,\infty)\times \mathbb{S}^{m-1}$, it holds:
\begin{equation}\label{eq:ordine}
J^\varrho(x,u) = \varrho^m J(x,\varrho u), \qquad \text{where}\quad J^\varrho(x,u) := \det\left(\frac{\partial E^\varrho}{\partial x},\frac{\partial E^\varrho}{\partial u}\right).
\end{equation}
Since $\Phi$ is smooth (or real-analytic), it follows immediately that the function $(x,\varrho,u)\mapsto J^\varrho(x,u)$ is continuous, smooth w.r.t.\ $\varrho$, and all the derivatives w.r.t.\ $\varrho$ are continuous. Furthermore, in the real-analytic case, $(x,\varrho,u)\mapsto J^\varrho(x,u)$ is real-analytic w.r.t.\ $\varrho$, locally uniformly w.r.t.\ $(x,u)$. It follows from \eqref{eq:ordine} and elementary analysis that
\begin{equation}
J(x, \varrho u) = \int_0^1 \de s_1 \int_0^{s_2} \de s_2\, \dots \int_0^{s_{m-1}} \de s_m \left.\frac{\partial^m J^t(x,u)}{\partial t^m}\right|_{t = \varrho s_m} .
\end{equation}
The claimed regularity follows immediately for the first factor in \eqref{eq:Theta}, while, for the second factor, it is a consequence of the smoothness (or real-analyticity) of $f$. This concludes the proof of the claim and shows the regularity of the volume of the tube w.r.t. $r\in [0,r_0)$.

We now characterize the Taylor expansion at $r=0$. It is clear that derivatives of order $k<m$ of \eqref{eq:98} vanish at $r=0$. While, for $k\geq m$, the Leibniz rule yields
\begin{equation}\label{eq:derivative}
\left.\frac{\de^k}{\de r^k}\right|_{r=0} \mu(T(r)) = \frac{(k-1)!}{(k-m)!} \int_{B\times \mathbb{S}^{m-1}} \left.\frac{\partial^{k-m}}{\partial r^{k-m}}\right|_{r=0} J(x,r u)f(E(x,r u)) \de x \de u.
\end{equation}
Note that for all $k\geq m$, by differentiating at $r=0$, it holds
\begin{equation}
\left.\frac{\partial^{k-m}}{\partial r^{k-m}}\right|_{r=0} J(x,r u)f(E(x,r u)) = (-1)^{k-m}\left.\frac{\partial^{k-m}}{\partial r^{k-m}}\right|_{r=0} J(x,-r u)f(E(x,-r u)),
\end{equation}
so that the right hand side of \eqref{eq:derivative} vanishes if $k-m$ is odd, upon integration on $\mathbb{S}^{m-1}$.

We show that \eqref{eq:sigmam} and \eqref{eq:wmj} are well-defined and then we prove \eqref{eq:expansion}. We use polar coordinates $(x,r,u) \in \mathrm{cl}(B)\times \R \times \mathbb{S}^{m-1}$ on $\pi^{-1}(\mathcal{O}\cap \mathrm{cl}(S))\cap A\tilde{S}$, noting that the slice $r=1$ corresponds to $A^1 S$. Note also that $\delta(E(x,ru)) = r$ and $E_*\partial_r = \nabla\delta$. Furthermore, in these coordinates $E_r(x,u) = E(x,ru)$. Thus
\begin{align}
\sigma_m = \lim_{r\to 0}\frac{E_r^*(\iota_{\nabla\delta} \mu)}{r^{m-1}} & = \lim_{r\to 0}\frac{J(x,ru)f(E(x,ru))}{r^{m-1}} r^{m-1} \iota_{\partial_r} \left( d r \wedge d x \wedge d u\right) \\
& = J(x,0)f(x,0) d x \wedge d u,\label{eq:sigmam_coords}
\end{align}
showing that $\sigma_m$ is a a well-defined and continuous density on $A^1 S$.

Similarly, by \cref{def:diverg} of iterated divergence we have
\begin{equation}
\diverg^{j}_{\mu/\delta^{m-1}}(\nabla\delta) \frac{\mu}{\delta^{m-1}} = 
\mathscr{L}^{j}_{\nabla\delta}\left(\frac{\mu}{\delta^{m-1}}\right), \qquad \text{on } U\setminus S.
\end{equation}
Taking the pull-back with $E$ we obtain in polar coordinates
\begin{multline}
\left(\diverg^{j}_{\mu/\delta^{m-1}}(\nabla\delta) \circ E(x,ru) \right) J(x,ru)f(E(x,ru)) d r \wedge d x \wedge d u = \nonumber \\ 
\mathscr{L}_{\partial_r}^j \Big( J(x,ru)f(E(x,ru)) d r \wedge d x \wedge d u\Big)  =\frac{\partial^j}{\partial r^{j}}J(x,ru)f(E(x,ru)) d r \wedge d x \wedge d u.
\end{multline}
Therefore we obtain in these coordinates
\begin{equation}\label{eq:wjm_coords}
w_m^{(j)}(x,u) =\lim_{r\to 0} \diverg^{j}_{\mu/\delta^{m-1}}(\nabla\delta) \circ E_r(x,u) =\frac{1}{J(x,0)f(x,0)}\left. \frac{\partial^j}{\partial r^{j}}\right|_{r=0} J(x,ru)f(E(x,ru)).
\end{equation}
This shows that the $w_m^{(j)}$ are well-defined and continuous functions on $A^1S$.

Putting together \eqref{eq:sigmam_coords} and \eqref{eq:wjm_coords} in \eqref{eq:derivative} we obtain
\begin{equation}
    \left.\frac{\de^k}{\de r^k}\right|_{r=0} \mu(T(r)) = \frac{(k-1)!}{(k-m)!}  \int_{A^1 S}w_m^{(k-m)}\, \de \sigma_m,
\end{equation}
concluding the proof of the Taylor expansion \eqref{eq:expansion}.
\end{proof}

%\subsection{Steiner's formula for the volume of a half-tube}

We prove \cref{thm:Steiner-intro}, of which we recall the statement for the reader's ease. %We recall that it is a specialization of \cref{thm:weyl-intro} to the case where $S$ is two-sided, see \cref{sec:twosided}. 

\begin{thm}\label{thm:Steiner}
 In the setting of \cref{thm:weyl}, assume furthermore that $S$ is a two-sided hypersurface. Then the volume of the half-tube $r \mapsto \mu(T^+(r))$ is smooth (or real-analytic) on $[0,r_0)$.
 Moreover, $\mu(T^+(r))$ has the following Taylor expansion at $r=0$:
     \begin{equation} \label{eq:Steiner_formula}
        \mu(T^+(r))=\sum_{k\geq 1 }  \frac{1}{k!}\left(\int_{S} \diverg^{k-1}(\nabla\delta_{\mathrm{sign}}) \de \sigma\right) r^k,
    \end{equation}
    where $\mathrm{d}\sigma$ denotes the \sr perimeter measure on the hypersurface $S$.
\end{thm}
\begin{proof}
For proving the regularity of the volume of the half-tube, one can repeat verbatim the proof of \cref{thm:weyl}. The main difference is that $A^1S=\ann S\cap \{\sqrt{2H}=1\}$ now has (locally) two connected components. Using the coordinates $(x,p)$ we can write 
\begin{equation}
    \mu(T^+(r)) = \int_0^r \int_{B} J(x,p) f(E(x,p)) \de x \de p,\qquad\text{for }r<r_0. 
\end{equation}
Similarly to \eqref{eq:derivative}, this implies that
\begin{equation}\label{eq:derivative2}
\left.\frac{\de^k}{\de r^k}\right|_{r=0} \mu(T^+(r)) = \int_{B} \left.\frac{\partial^{k-1}}{\partial r^{k-1}}\right|_{r=0} J(x,r)f(E(x,r)) \de x.
\end{equation}
Since we are no longer integrating on a sphere, the odd coefficients do not vanish. Arguing as in the proof of \cref{thm:weyl} we rewrite \eqref{eq:derivative2} in terms of the functions $w_m^{(j)}$ and density $\sigma_m$:
\begin{equation}
    \left.\frac{\de^k}{\de r^k}\right|_{r=0} \mu(T^+(r)) = \int_{A^{1+}S} w^{(k-1)}_1 \de \sigma_1,
\end{equation}
where $A^{1+}S =A^+S \cap A^1S$ is the positive component of the unit annihilator bundle. Note that $\pi:AS \to S$ restricts to a diffeomorphism between $A^{1+}S$ and $S$. Under this identification, we identify $w^{(j)}_1$ with the continuous function $\diverg^{j-1}_{\mu}(\nabla\delta_{\mathrm{sign}})$ on $S$. Similarly, we identify the density $\sigma_1$ with the continuous density $\sigma = \iota_{\nabla\delta} \mu$ on $S$. Formula \eqref{eq:Steiner_formula} is then proved.
\end{proof}

\section{Polynomial character of iterated divergences}
\label{sec:polynomials}

In this section, we investigate to what extent the coefficients of the Weyl's tube formula are \emph{polynomial} functions in the derivatives of the distance $\delta$ from the submanifold. 
In particular, we prove the general version of \cref{thm:div_poly-intro}, as well as \cref{prop:forsedatogliere-intro}.

\begin{lem}
\label{lem:iter_div_poly}
   Let $M$ be a smooth (or real-analytic) \sr manifold, equipped with a smooth (or real-analytic) measure $\mu$. Let $Y_1,\ldots,Y_n$ be a local smooth (or real-analytic) frame for $TM$ defined on a neighbourhood $\mathcal O\subset M$. Then, for all $k \in \N$ there exists a polynomial function $P^k$ in $n+n^2$ variables, with smooth (or real-analytic) coefficients on $\mathcal{O}$, homogeneous of degree $k$, such that for any $C^2$ non-characteristic submanifold $S\subset M$ of codimension $m\geq 1$, it holds 
    \begin{equation}
    \label{eq:Pk}
        \diverg^k_{\mu}(\nabla\delta)=P^k\left(\dots,Y_\alpha\delta,\dots,Y_{\alpha}Y_{\beta}\delta,\dots\right),\qquad \text{on } \mathcal{O}\cap U\setminus S,
    \end{equation}
    where in the variables $Y_\alpha\delta$, $Y_\alpha Y_\beta\delta$ the indices $\alpha,\beta$ run over the set $\{1,\dots,n\}$. If $S$ is a two-sided hypersurface, then \eqref{eq:Pk} for $m=1$ holds on the whole $U$, replacing $\delta$ with $\delta_{\mathrm{sign}}$.
\end{lem}

\begin{proof}
    We proceed by induction on $k$. Let $\{X_1,\ldots,X_N\}$ be a smooth (or real-analytic) generating family for the \sr structure. 
    In the following, we consider the indices $\alpha,\beta,\varrho\in \{1,\ldots,n\}$ and $h \in \{1,\ldots,N\}$ and we adopt the convention for which repeated indices are summed over the corresponding range.
    We have that $X_h=a^\alpha_{h}Y_\alpha$ for some smooth (or real-analytic) functions $a^\alpha_{h}$ defined on $\mathcal O$.
    Hence,
    \begin{equation}
        \nabla\delta= (X_h\delta)X_h= a^\alpha_{h}a^\beta_{h}(Y_\alpha\delta )Y_\beta.
    \end{equation}
    The case $k=0$ is trivial since $\diverg^0_{\mu}(\nabla\delta)=1$.
    For the case $k=1$ we have that 
      \begin{align}\label{eq:P0}
     \diverg_{\mu}^1(\nabla\delta)&=(X_hX_h\delta)+(X_h\delta)\diverg_\mu(X_h)\\
        &= 
        a^\alpha_{h}a^\beta_{h}\left[(Y_\alpha Y_\beta\delta)+(Y_\alpha\delta)\diverg_\mu (Y_\beta)\right]+a^\alpha_{h}(Y_\alpha a^\beta_{h})(Y_\beta\delta)+a^\alpha_{h}(Y_\beta a^\beta_{h})(Y_\alpha\delta).
    \end{align}
Thus, $\diverg_{\mu}^1(\nabla\delta)$ is a homogeneous polynomial of degree $1$ with smooth (or real-analytic) coefficients in the variables $Y_\alpha\delta,Y_\alpha Y_\beta \delta$, that we denote $P^1$.
    
    Let us suppose that the statement of the lemma is true for some $k\geq 1$, i.e.\ 
    \begin{equation}
        \diverg^k_{\mu}(\nabla\delta)= P^k\left(\dots,Y_\alpha\delta,\dots,Y_{\alpha}Y_{\beta}\delta,\dots\right).
    \end{equation}
    We prove the statement for $\diverg^{k+1}_{\mu}(\nabla\delta)$.
    Exploiting the recursive relation in \eqref{eq:recursive_relation}, we obtain that
    \begin{equation}
    \label{eq:Pk+1}
        \diverg^{k+1}_{\mu}(\nabla\delta)
%        =\diverg_{\mu}(\nabla\delta)\, \diverg^k_{\mu}(\nabla\delta) +\nabla\delta(\diverg^k_{\mu}(\nabla\delta))
        =P^1 P^k  +\nabla\delta(P^k ),
    \end{equation}
 where we omit the explicit dependence on the variables. We claim that \eqref{eq:Pk+1} is homogeneous of degree $k+1$.  By \eqref{eq:P0} and the inductive hypothesis, $P^1P^k$ satisfies the claim. For $\nabla\delta(P^k)$,
    by the Leibniz rule, it is sufficient to show that, for any fixed $i,j=1,\ldots,n$,
    \begin{equation}\label{eq:variables}
                \nabla\delta(Y_i\delta) 
                  \qquad \text{and} \qquad
        \nabla\delta(Y_iY_j\delta)
    \end{equation}
    are homogeneous polynomials of degree $2$ in the variables $Y_\alpha\delta,Y_\alpha Y_\beta \delta$, where $\nabla\delta(Y_iY_j\delta)$ has to be intended in the distributional sense. The claim for $\nabla\delta(Y_i\delta)$ is immediate, since it holds 
    \begin{equation}
        \nabla\delta(Y_i\delta) 
        =  a^\alpha_{h}a^\beta_{h}(Y_\alpha\delta)(Y_\beta Y_i\delta).
    \end{equation}
    To deal with the last item in \eqref{eq:variables}, consider the bracket relations 
    $[Y_\alpha,Y_\beta]=c^\varrho_{\alpha\beta}Y_\varrho$, where $c^\varrho_{\alpha\beta}: \mathcal O \to \R$ are smooth (or real-analytic) We obtain the following distributional identity:
    \begin{equation}
    \label{eq:three_der}
    \begin{split}
        \nabla\delta(Y_iY_j\delta)        &=   a^\alpha_{h}a^\beta_{h}
            (Y_\alpha\delta)(Y_\beta Y_iY_j\delta) \\ 
        &=   a^\alpha_{h}a^\beta_{h}
            (Y_\alpha\delta)\left[(Y_iY_\beta Y_j\delta)
            +c^\varrho_{\beta i}(Y_\varrho Y_j\delta) \right]\\
        &=   a^\alpha_{h}a^\beta_{h}
            (Y_\alpha\delta)(Y_iY_jY_\beta\delta) 
        %\\ &\qquad  \qquad  \quad
        +a^\alpha_{h}a^\beta_{h}(Y_\alpha\delta)\left[Y_i\left(c^\varrho_{\beta j}(Y_\varrho\delta)\right)
                +c^\varrho_{\beta i}(Y_\varrho Y_j \delta)\right]. 
    \end{split}
    \end{equation}
    To get rid of the term containing the third order derivative of $\delta$, we exploit the Eikonal equation $\|\nabla \delta\|=1$, valid on $U\setminus S$. More precisely, differentiating it twice along the vector fields $Y_i,Y_j$ we obtain, in the sense of distributions,
    \begin{equation}
    \label{eq:diffEikonal}
        %0=Y_hY_j\left( \frac{1}{2} \sum_{i=1}^{N} (Y_i\delta)^2 \right)=
        Y_iY_j(\|\nabla\delta\|^2)=0\qquad\implies\qquad(X_h\delta)(Y_iY_jX_h\delta)+ (Y_iX_h\delta)(Y_jX_h\delta)
        =0.
    \end{equation}
    Hence, by a direct computation, from \eqref{eq:diffEikonal} we deduce that 
    \begin{equation}
         0=  a^\alpha_{h}a^\beta_{h}
         (Y_\alpha\delta)(Y_iY_jY_\beta\delta)+ \tilde P,
    \end{equation}
    where $\tilde P$ is a suitable homogeneous polynomial of degree $2$ in the variables $Y_\alpha\delta$, $Y_\alpha Y_\beta\delta$, with smooth (or real-analytic) coefficients on $\mathcal{O}$.
    Therefore, the distribution $ a^\alpha_{h}a^\beta_{h}(Y_\alpha\delta)(Y_iY_jY_\beta\delta)$ is actually equal to the continuous function $\tilde P\left(\dots,Y_\alpha\delta,\dots,Y_{\alpha}Y_{\beta}\delta,\dots\right)$. 
    Then, substituting the latter expression in \eqref{eq:three_der}, we conclude the proof.  
\end{proof}

By using the Leibniz rule we obtain the following formula for all $m\geq 1$:
    \begin{equation}
        \diverg_{{\mu}/{\delta^{m-1}}}^k(\nabla\delta)=\sum_{j=0}^k c_{j,k,m} \frac{\diverg_\mu^{k-j}(\nabla\delta)}{\delta^j},
    \end{equation}    
    for explicit coefficients $c_{j,k,m}\in \R$.
%    \begin{equation}
%     c_{j,k,m}:=\begin{cases} \binom{k}{j}(-1)^j \prod_{i=0}^{j-1} (m-1+i) & j \geq 1 \\
%        1 & j=0
%        \end{cases}
%    \end{equation}
Then we obtain the following generalization of \cref{lem:iter_div_poly}.

\begin{cor}
\label{cor:iter_div_poly_m}
   Let $M$ be a smooth (or real-analytic) \sr manifold, equipped with a smooth (or real-analytic) measure $\mu$. Let $Y_1,\ldots,Y_n$ be a local smooth (or real-analytic) frame for $TM$ defined on a neighbourhood $\mathcal{O}\subset M$. Then, for all $k,m \in \N$ with $m\geq 1$, there exists a polynomial function $P^k_m$ in $n+n^2+1$ variables, with smooth (or real-analytic) coefficients on $\mathcal{O}$, homogeneous of degree $k$, such that for any $C^2$ non-characteristic submanifold $S\subset M$ of codimension $m\geq 1$, it holds 
    \begin{equation}
    \label{eq:Pk_m}
        \diverg^k_{\mu/\delta^{m-1}}(\nabla\delta)=P^k_m\left(\dots,Y_\alpha\delta,\dots,Y_{\alpha}Y_{\beta}\delta,\dots,\tfrac{m-1}{\delta}\right),\qquad \text{on } \mathcal{O}  \cap U\setminus S,
    \end{equation}
    where in the variables $Y_\alpha\delta$, $Y_\alpha Y_\beta\delta$ the indices $\alpha,\beta$ run over the set $\{1,\dots,n\}$. If $S$ is a two-sided hypersurface, then \eqref{eq:Pk_m} for $m=1$ holds on the whole $U$, replacing $\delta$ with $\delta_{\mathrm{sign}}$.
\end{cor}
\begin{rmk}
The coefficient $m-1$ in the last variable of \eqref{eq:Pk_m} is a useful notation to recall that, if $m=1$, there is no dependence on that variable.
\end{rmk}
\begin{rmk}[The left-invariant case]\label{rmk:invariance}
In case $M$ is a Lie group equipped with a left-invariant sub-Riemannian structure and a left-invariant measure, one can choose in \cref{lem:iter_div_poly} and \cref{cor:iter_div_poly_m} as $Y_1,\dots,Y_n$ a left-invariant global frame, adapted to the sub-Riemannian distribution. In this case the $P^k$'s (resp.\ the $P^k_m$) are polynomials with constant coefficients, canonically associated with the sub-Riemannian structure once a left-invariant frame is fixed.
\end{rmk} 

We now prove \cref{prop:forsedatogliere-intro}, of which we recall the statement for convenience. We refer to \cref{sec:invariance} for a definition of the Heisenberg group.

\begin{prop}
\label{prop:forsedatogliere}
Let $\mathbb{H}$ be the three-dimensional Heisenberg group, equipped with a left-invari\-ant measure $\mu$. Then, there exists polynomials $Q^k_{m}$ with real coefficients and in $5$ variables, such that for any $C^2$ non-characteristic submanifold $S\subset M$ with codimension $m \in \{1,2\}$ it holds
\begin{equation}
\label{eq:lasola}
\diverg^{k}_{\mu/\delta^{m-1}}(\nabla\delta)= Q^k_m\left(F_1,F_2,F_3,F_4,\tfrac{m-1}{\delta}\right),\qquad \text{on } U\setminus S,
\end{equation}
where, for any given left-invariant and oriented orthonormal frame $\{X_1,X_2\}$, and letting $X_0$ be the Reeb vector field, we define the functions $F_i : U\setminus S \to \R$ by
    \begin{align}
        F_1 & :=X_1X_1\delta+X_2X_2\delta,  & F_2 & :=-(X_2\delta)(X_1 X_0\delta)+ (X_1\delta)(X_2 X_0\delta),\\ 
        F_3 & :=X_0\delta,  & F_4 & :=X_0 X_0\delta.
    \end{align}
   (The functions $F_1,\dots,F_4$ do not depend on the choice of the frame.) 
 Furthermore, if $S$ is a two-sided surface, \eqref{eq:lasola} for $m=1$ holds on the whole $U$, replacing $\delta$ with $\delta_{\mathrm{sign}}$.
\end{prop}
\begin{proof}
Fix a left-invariant orthonormal oriented frame $\{X_1,X_2\}$ for $\mathbb{H}$ and let $X_0$ be the Reeb vector field (which is also left-invariant). Let $P^k_m$ be the polynomials of \cref{cor:iter_div_poly_m}. By \cref{rmk:invariance}, these are polynomials with constant real coefficients such that
\begin{equation}
\diverg^k_{\mu/\delta^{m-1}}(\nabla\delta)= P^k_m\left(\dots,X_\alpha\delta,\dots,X_{\alpha}X_\beta\delta\dots,\tfrac{m-1}{\delta}\right),
\end{equation}
where $\alpha,\beta=0,1,2$. We employ the shorthand $X_{\alpha\beta}:=X_\alpha X_\beta$. It is useful to introduce the following vectors and matrices:
\begin{equation}
\xi = X_i\delta, \quad z = X_0\delta,\quad A = X_i X_j \delta, \quad v = X_0 X_j \delta = X_jX_0\delta, \quad \alpha = X_0 X_0\delta,
\end{equation}
for $i,j=1,2$, so that the variables appearing in $P^k_m$ can be organized as follows:
    \begin{equation}
\left(            \begin{pmatrix}
                \xi\\
                z
            \end{pmatrix};
            \begin{pmatrix}
                A & v\\
                v^t & \alpha
            \end{pmatrix}; \frac{m-1}{\delta}
\right).
    \end{equation}
It is well-known that the group of isometries (fixing the origin) of $\mathbb{H}$ consists in all group automorphisms that act by orthogonal transformations on the distribution, \cite{Ursula,LD-Isometries}. We refer to \cite[Sec.\ 6]{LR-howmany} for the explicit description we employ here. It is isomorphic to $\mathrm{SO}(2)$, and any isometry is determined by its action on the distribution at the origin. Let $\phi_M : \mathbb{H}\to\mathbb{H}$ be such an isometry for some $M\in \mathrm{SO}(2)$. The corresponding action on a left-invariant oriented orthonormal frame $X_1,X_2$ is
\begin{equation}
(\phi_M)_* X_i = \sum_{j=1}^2 M_{ij} X_j,\qquad (\phi_M)_*  X_0 =  X_0.
\end{equation}
It is straightforward to show that the polynomials $P^k_m$ of \cref{cor:iter_div_poly_m} are invariant under the action of isometries of $\mathbb{H}$ in the following sense: for all $M\in \mathrm{SO}(2)$ it holds
\begin{equation}\label{eq:rotation}
P^k_m\left(
        \begin{pmatrix}
            M\xi\\
             z
        \end{pmatrix};
        \begin{pmatrix}
            M A M^t & Mv\\
            (Mv)^t &  \alpha
        \end{pmatrix}
        ; \frac{m-1}{\delta}\right)
        = 
        P^k_m\left(\begin{pmatrix}
                \xi\\
                z
            \end{pmatrix};
            \begin{pmatrix}
                A & v\\
                v^t & \alpha
            \end{pmatrix}; \frac{m-1}{\delta}
\right).
\end{equation}
Note that by the Eikonal equation $\|\xi\|^2=1$ and, by differentiating it, we find:
    \begin{equation}
    \label{eq:relations}
        A\xi=0\qquad\text{and}\qquad v^t\xi=0.
    \end{equation}
Furthermore we have
\begin{equation}\label{eq:symmetry}
A_{ij} - A_{ji} = [X_i,X_j]\delta = z J_{ij},\qquad \text{where}\qquad  J=\begin{pmatrix}
        0& 1\\
        -1 & 0
        \end{pmatrix}.
\end{equation}    
Thus, for fixed $p\in \mathbb{H}$, let $M \in \mathrm{SO}(2)$ be the unique orthogonal matrix such that $M\xi|_p = e_1$ and $M J\xi|_p = e_2$. Using relations \eqref{eq:relations}, we obtain that $Mv=(0, (J\xi)^tv)$ and  
    \begin{equation}
    \label{eq:normal_form}
        MAM^t=\begin{pmatrix}
            0 & z\\
            0 & (J\xi)^t A J\xi
        \end{pmatrix}=\begin{pmatrix}
            0 & z\\
            0 & \mathrm{Tr}(A)
        \end{pmatrix}.
    \end{equation}
By \eqref{eq:rotation} we deduce that there exist polynomials $Q_k^m$ with constant coefficients such that
\begin{equation}
\diverg^k_{\mu/\delta^{m-1}}(\nabla\delta) =Q_m^k\left(\mathrm{Tr}(A),(J\xi)^t v,z,\alpha,\tfrac{m-1}{\delta}\right),
\end{equation}
as functions on $U\setminus S$.
\end{proof}

The strategy of the proof of \cref{prop:forsedatogliere} can work, in principle, also for Heisenberg groups of a higher dimension. However, for $\mathbb{H}_{2d+1}$, the isometry group (fixing the identity, and -- for simplicity -- preserving the orientation) is $\mathrm{O}(2d)\cap \mathrm{Sp}(2d) \simeq \mathrm{U}(d)$. In particular, while $\mathrm{U}(1) \simeq \mathrm{SO}(2)$, in higher dimension this is no longer true, and in particular, the polynomials $P^k$ are no longer invariant by the action of the full (special) orthogonal group. Thus one cannot put the (symmetric part of) the matrix $A=(X_iX_j\delta)_{i,j=1}^{2d}$ in normal form, as done in \cref{prop:forsedatogliere}. This is one of the reasons for which the generalization of \cref{prop:forsedatogliere} to higher-dimensional Heisenberg groups seems difficult.

\section{Weyl's invariance for curves in the Heisenberg groups}
\label{sec:invariance}

In this section, we prove a suitable Weyl's invariance theorem for curves in Heisenberg groups. We briefly recall their definition: consider the Lie group $(\R^{2d+1},\star)$, where 
\begin{equation}
    (x,y,z)\star (x',y',z')=\left(x+x',y+y',z+z'+\frac{1}{2}(x\cdot y'-x'\cdot y)\right),
\end{equation}
for every $(x,y,z),(x',y',z')\in \R^d\times\R^d\times\R\cong\R^{2d+1}$. Then, the $(2d+1)$-dimensional Heisenberg group $\mathbb{H}_{2d+1}$ is defined by the \sr structure on $\R^{2d+1}$ given by the global generating frame $\{X_1,\ldots,X_{2d}\}$ of left-invariant vector fields:
\begin{equation}
X_{i}=\partial_{x_i}-\frac{y_i}{2}\partial_z,\qquad X_{d+i}=\partial_{y_i}+\frac{x_i}{2}\partial_z, \qquad i=1,\ldots,d .
\end{equation}
 The \sr metric $g$ is such that $X_1,\ldots,X_{2d}$ are orthonormal. Moreover, it holds
\begin{equation}
[X_i,X_j] = J_{ij} X_0, \qquad J:=\begin{pmatrix}
0 &  \mathbbm{1} \\
-\mathbbm{1} &  0 
\end{pmatrix}, \qquad \forall\, i,j=1,\dots,2d,
\end{equation}
where $X_0 = \partial_z$ is the \emph{Reeb vector field}. Note that $[X_0,X_j]=0$ for all $j=1,\dots,2d$.

On $\mathbb{H}_{2d+1}$, we also consider the canonical Riemannian extension of the \sr metric. Namely, we define the left-invariant Riemannian metric $g_R$ extending $g$ by declaring the Reeb vector field to be of norm $1$ and orthogonal to the Heisenberg distribution. 

% We refer to \cref{def:reebangle} for the definition of \emph{Reeb angle}

We first observe that the \emph{Reeb angle} (see \cref{def:reebangle}) characterizes suitable equivalence classes of curves in the Heisenberg groups. 

\begin{lem}\label{lem:diffeo}
Let $\gamma_i:[0,L_i]\to \mathbb{H}_{2d+1}$, $i=1,2$, be non-characteristic $C^2$ curves, parametrized with unit Riemannian speed. Denote with $\Gamma_i \subset \mathbb{H}_{2d+1}$ the corresponding embedded submanifold. Then the following are equivalent:
\begin{enumerate}[(i)]
\item\label{item:diffeo1} $\Gamma_1$ and $\Gamma_2$ have the same Riemannian length and the same Reeb angle, i.e.\ $L_1=L_2=:L$, and 
 $\theta_{\Gamma_1}(\gamma_{1}(t))=\theta_{\Gamma_2}(\gamma_2(t))$ for all $t\in [0,L]$;
\item\label{item:diffeo2} there exists a diffeomorphism $\phi: \Gamma_1\to \Gamma_2$ such that, for all $q\in \Gamma_1$, there is a smooth sub-Riemannian isometry\footnote{A \emph{smooth sub-Riemannian isometry} is a diffeomorphism $\Phi: M \to M$ such that $\Phi_*$ is an orthogonal transformation on the sub-Riemannian distribution or, equivalently, the Hamiltonian satisfies $H\circ \Phi^* = H$.} $\Phi: \mathbb{H}_{2d+1} \to \mathbb{H}_{2d+1}$ such that $\phi_* = \Phi_*|_{T_q\Gamma_1}$.
\end{enumerate}
\end{lem}
\begin{rmk}
\cref{item:diffeo2} is weaker than asking that $\Gamma_1,\Gamma_2$ are diffeomorphic by a sub-Riemannian isometry. In fact, the isometry $\Phi$ can depend on $q\in \Gamma_1$.
\end{rmk}
\begin{proof}
We prove \ref{item:diffeo1} $\Rightarrow$ \ref{item:diffeo2}.
Define the diffeomorphism $\phi : \Gamma_1 \to \Gamma_2$ by $\phi = \gamma_2 \circ \gamma^{-1}_1$. Decompose
\begin{align}
\dot\gamma_i(t) = W_i|_{\gamma_i(t)} + \alpha_i(t) X_0|_{\gamma_i(t)}, \qquad i=1,2,
\end{align}
for some never-vanishing $\alpha_i : [0,L_i] \to \R$. Since the two curves have the same Riemannian length and Reeb angle  it holds $L_1=L_2=:L$ and $\alpha_1 = \pm \alpha_2$. It follows that $\|W_1\| = \|W_2\|$.

Fix $t\in [0,L]$. Up to applying a left-translation (which is an isometry), we may assume that $\gamma_1(t) = \gamma_2(t) = e$ is the Heisenberg group identity. Let $\Phi$ be a sub-Riemannian isometry such that $\Phi_*$ sends $W_1|_e$ to $W_2|_e$ and $X_0$ to  $\pm X_0$. (Such an isometry exists by the structure of Heisenberg isometries, see e.g. \cite[Sec.\ 6.1]{LR-howmany}.) Then, by construction
\begin{align}
\phi_* (\dot\gamma_1(t))  = \dot\gamma_2(t)  & = W_2|_{e} + \alpha_2(t) X_0|_{e}\\
&  = W_2|_{e} \pm \alpha_1(t)  X_0|_{e}\\
&  = \Phi_*\left( W_1|_{e} + \alpha_1(t)  X_0|_{e}\right)\\
& = \Phi_*(\dot\gamma_1(t)).
\end{align}
% \begin{equation}
%     \phi_* (\dot\gamma_1(t))  = \dot\gamma_2(t)   = W_2|_{e} + \alpha_2(t) X_0|_{e}=
%     W_2|_{e} \pm \alpha_1(t) X_0|_{e}=
%     \Phi_*\left( W_1|_{e}\right) + \alpha_1(t) \Phi_*\left( X_0|_{e}\right)=
%     %\Phi_*\left( W_1|_{e}\right) \pm \alpha_1(t) X_0|_{e}
%     \Phi_*(\dot\gamma_1(t)).
% \end{equation} 

We prove \ref{item:diffeo2} $\Rightarrow$ \ref{item:diffeo1}. Define the diffeomorphism $\zeta:=\gamma_2^{-1}\circ\phi\circ \gamma_1:[0,L_1]\to [0,L_2]$. Observe that sub-Riemannian isometries of $\mathbb{H}_{2d+1}$ are also Riemannian isometries of the corresponding Riemannian extension. Using this fact and \cref{item:diffeo2} one easily sees that $\dot{\zeta} = \pm 1$. Without loss of generality we can assume that $\dot{\zeta}=1$ (otherwise, re-parametrize one of the two curves going backwards). Thus, $L_1=L_2=:L$ and $\phi\circ \gamma_1(t) = \gamma_2(t)$ for all $t\in [0,L]$. Furthermore, for all fixed $t\in [0,L]$, letting $\Phi$ be the isometry of \cref{item:diffeo2} at the point $q=\gamma_1(t)$, it holds $\phi_*\dot\gamma_1(t) = \Phi_*\dot\gamma_1(t) = \dot{\gamma}_2(t)$. Using \cref{def:reebangle} and the fact that $\Phi_*X_0=\pm X_0$, we see that
\begin{align}
\theta_{\Gamma_2}(\gamma_2(t)) & = |g_R(\dot\gamma_2(t), X_0|_{\gamma_2(t)})| \\
& = |g_R(\Phi_*\dot\gamma_1(t),\Phi_* X_0|_{\gamma_1(t)})| \\
& = |g_R(\dot\gamma_1(t), X_0|_{\gamma_1(t)})| \\
& = \theta_{\Gamma_1}(\gamma_1(t)).
\end{align}
Since $t$ is arbitrary, the proof is concluded.
\end{proof}

We now prove \cref{thm:invariance-intro}, which states that the volume of small sub-Riemannian tubes around a non-characteristic curve in $\mathbb{H}_{2d+1}$ depend only on the Reeb angle and the Riemannian length of the curve. We recall the statement for convenience.

\begin{thm}\label{thm:invariance}
Let $\gamma,\gamma':[0,L]\to \mathbb{H}_{2d+1}$, be non-characteristic $C^2$ curves, parametrized with unit Riemannian speed. Denote with $\Gamma,\Gamma' \subset \mathbb{H}_{2d+1}$ the corresponding embedded submanifold. Assume that $\theta_{\Gamma}(\gamma_t)=\theta_{\Gamma'}(\gamma'_t)$ for all $t\in [0,L]$. Then, there exists $\epsilon>0$ such that
\begin{equation}
\mu(T_{\Gamma}(r)) = \mu(T_{\Gamma'}(r)), \qquad \forall \, r \in[0,\epsilon),
\end{equation}
where $\mu$ denotes the Lebesgue measure of $\mathbb{H}_{2d+1}$.
\end{thm}
\begin{proof}
For $q\in \Gamma$ and $V\in T_q \Gamma$, we use the symbol $\tilde{V}$ to denote a lift on $\ann \Gamma$, namely a smooth vector field on $\ann \Gamma$, such that $\pi_* \tilde{V}|_\lambda = V$ for all $\lambda \in \pi^{-1}(q)\cap AS$. Also, set $M= \mathbb{H}_{2d+1}$ so that $n=\dim M = 2d+1$. We denote with a slight abuse of notation $E:T^*M \to M$ the global exponential map, defined by the same formula \eqref{eq:normal_exponential_map}. (In the rest of the paper, this notation was employed for the restriction to $A\Gamma$. No confusion should arise.)

For $V\in T_q\Gamma$, we define a $(n-1)$-form on the $(n-1)$-dimensional submanifold $\ann_q\Gamma\subset \ann \Gamma$ (the fiber at $q$) by the position
\begin{equation}\label{eq:definitionoffiberdensity}
\xi_1,\dots,\xi_{n-1}\mapsto \iota_{\tilde{V}}(E^*\mu)(\xi
_1,\dots,\xi_{n-1}):=(E^*\mu)(\tilde{V},\xi_1,\dots,\xi_{n-1}),
\end{equation}
for all $\xi_1,\dots,\xi_{n-1}\in T_{\lambda}(\ann_q\Gamma)$, and all $\lambda \in \ann_q\Gamma$. The definition is well-posed in the sense that it does not depend on the choice of the lifts, but only on the original $V$.

We need to fix orientations. Let $\epsilon =\min\{r_0(\Gamma),r_0(\Gamma')\}$. The standard orientation on $M=\mathbb{H}_{2d+1}$ induces a unique orientation on $A\Gamma$ such that the map $E$ is an orientation-preserving diffeomorphism when restricted on $A\Gamma \cap \{\sqrt{2H}<\epsilon\}$. Also, fix an orientation on $\Gamma$ in such a way that the embedding $\gamma : [0,L] \to \Gamma$ is orientation-preserving. Let $\psi: U\times \R^{n-1} \to \pi^{-1}(U)\subset A\Gamma$ be an orientation-preserving trivialization (where on $\R^{n-1}$ we choose the standard orientation and on $U\subset \Gamma$ the induced orientation). This induces also an orientation on each fiber $A_q \Gamma$ in such a way that $\psi|_{A_q \Gamma} : \R^{n-1} \to A_q \Gamma$ is an orientation-preserving diffeomorphism. The same choices are done for $\Gamma'$. The integrations below are carried out using these choices.

The integration of the $(n-1)$-form \eqref{eq:definitionoffiberdensity} on the disk of radius $r$ in $\ann_q\Gamma$ yields a one-parameter family of $1$-forms on $\Gamma$, denoted by $\mu_\Gamma^r$, defined as follows:
\begin{equation}\label{eq:defofform}
\mu_\Gamma^r|_q(V)=\int_{\ann_q\Gamma \cap \{\sqrt{2H}<r\}} \iota_{\tilde{V}}(E^*\mu),
\end{equation}
for all $V\in T_q \Gamma$, and all $q\in \Gamma$, $r\in [0,\epsilon)$, for $\epsilon$ sufficiently small. An application of the Fubini theorem with the above choice of orientations yields
\begin{equation}\label{eq:tuboemisure}
\mu(T_\Gamma(r)) = \int_\Gamma \mu_\Gamma^r, \qquad \forall \, r\in [0,\epsilon).
\end{equation}
Let $\phi:\Gamma\to \Gamma'$ be the diffeomorphism from \cref{item:diffeo2} of \cref{lem:diffeo}. We claim that $\phi^*\mu_{\Gamma'}^r = s_\phi \mu_\Gamma^r$ (where $s_\phi = \pm 1$ depending on the fact that $\phi$ is orientation-preserving or orientation-reversing). We now prove this claim, which implies the theorem, by \eqref{eq:tuboemisure}.

%We start with a preliminary but key observation. (In the following, for ease of notation, $M=\mathbb{H}_{2d+1}$.) If $O:T^*M\to T^*M$ is a smooth bundle isomorphism such that $O|_{\ann \Gamma}$ preserves $H$ (i.e., $O$ is an orthogonal transformation in each fiber of $\ann \Gamma$, equipped with the scalar product induced by $H$) then the following identity holds as densities on $\ann_q\Gamma$:
%\begin{equation}\label{eq:reparaminvariance}
%O^*\left(\iota_{\tilde{V}}(E^*\mu)\right) = \iota_{O^{-1}_*\tilde{V}}((E\circ O)^*\mu) = \iota_{\tilde{V}}((E\circ O)^*\mu),
%\end{equation}
%where, in the last equality, we used the fact that $O^{-1}_*\tilde{V}$ is a vector field tangent to $\ann \Gamma$ projecting to $V \in T_q \Gamma$. Therefore, in \eqref{eq:defofform}, we can replace $E$ with $E\circ O$, for an arbitrary bundle map $O$ satisfying the above assumptions. In other words, we are free to change the exponential map $E$ by a ``twisted'' one $E\circ O$. We remark that the twist is only in the directions transverse to $\ann \Gamma\subseteq T^*M$.

Fix $q\in \Gamma$ and set $q':=\phi(q)\in \Gamma'$. Let $\Phi:M\to M$ be the isometry such that $\Phi_*|_{T_q \Gamma} = \phi_*:T_q \Gamma \to T_{q'}\Gamma'$, which exists by \cref{lem:diffeo}. Note that $\Phi^{{-1}*}:T^*M\to T^*M$ (the pull-back of the inverse) restricts to a diffeomorphism between the corresponding annihilators $\ann_q\Gamma$ onto $A_{q'}\Gamma'$, and it preserves the Hamiltonian. In particular
\begin{equation}\label{eq:mapdomain}
A_{q'}\Gamma' \cap \{\sqrt{2H}<r\} = \Phi^{-1*} (A_{q}\Gamma \cap \{\sqrt{2H}<r\}),\qquad \forall\, r\in [0,\epsilon).
\end{equation}
Note that $\Phi: M\to M$ is a $s_{\Phi}$-orientation-preserving diffeomorphism, with $s_{\Phi} = \pm 1$ and recall also that $\phi : \Gamma\to \Gamma'$ is $s_{\phi}$-orientation-preserving. It follows that $\Phi^{{-1}*}|_{ A_q \Gamma} : A_q \Gamma \to A_{q'} \Gamma'$ is $(s_{\Phi} s_{\phi})$-orientation-preserving. Thus, for $V\in T_q \Gamma$, it holds
\begin{align}
(\phi^*\mu_{\Gamma'}^r)|_q(V)  & = \mu_{\Gamma'}^r|_{q'}(\phi_* V) \\
& = \mu_{\Gamma'}^r|_{q'}(\Phi_* V) \\
& = \int_{A_{q'}\Gamma' \cap \{\sqrt{2H}<r\}} \iota_{\widetilde{\Phi_*V}}(E^*\mu) &\text{by definition of $\mu_{\Gamma'}^r$} \\
& = \int_{\Phi^{{-1}*}(A_{q}\Gamma \cap \{\sqrt{2H}<r\})} \iota_{\widetilde{\Phi_*V}}(E^*\mu) & \text{by \eqref{eq:mapdomain}}\label{eq:lastine-1} \\
& = s_{\Phi} s_{\phi} \int_{A_{q}\Gamma \cap \{\sqrt{2H}<r\}} \Phi^{{-1}**}\iota_{\widetilde{\Phi_*V}}(E^*\mu). \label{eq:lastline}
\end{align}
%Here $O':T^*M\to T^*M$ is an arbitrary bundle map with $O'|_{\ann \Gamma'}$ fiber-wise orthogonal. 
For clarity, we stress that $\Phi^{{-1}*}:T^*M \to T^*M$, while $\Phi^{{-1}**}: T^*(T^*M)\to T^*(T^*M)$, and $(\Phi^{{-1}*})_*:T(T^*M) \to T(T^*M)$. Observe that $\Phi^{-1*}:T^*M\to T^*M$ is such that $\pi \circ \Phi^{-1*} = \Phi \circ \pi$ as maps on $T^*M$. It follows that
\begin{equation}
\pi_* \circ (\Phi^{-1*})_* \tilde{V} = \Phi_* V.
\end{equation}
If $(\Phi^{-1*})_* \tilde{V}$ were tangent to $\ann \Gamma'$, then it would be a smooth vector field on $\ann \Gamma'$ that lifts $\Phi_*V\in T_{q'}\Gamma'$, so that we could write, according to our notation:
\begin{equation}\label{eq:liftandisofalse}
(\Phi^{-1*})_* \tilde{V} = \widetilde{\Phi_*V}.
\end{equation}
However, even though $\Phi^{-1*}$ sends $\ann_q\Gamma$ to $A_{q'}\Gamma'$, it does not restrict to a map from $\ann \Gamma$ to $\ann \Gamma'$. Therefore \eqref{eq:liftandisofalse} is false in general but rather there exists a vector field $\zeta$ on $T^*M$ that is vertical (i.e.\ tangent to the fibers of $T^*M$, but possibly not to the fibers of $\ann \Gamma$), such that
\begin{equation}\label{eq:liftandiso}
(\Phi^{-1*})_* (\tilde{V} +\zeta) = \widetilde{\Phi_*V}.
\end{equation}
Therefore, for the density integrated in \eqref{eq:lastline} we have
\begin{align}
\Phi^{{-1}**}\iota_{\widetilde{\Phi_*V}}(E^*\mu) 
& = \iota_{\tilde{V}+\zeta}(\Phi^{{-1}**} \circ E^*\mu) & \text{by \eqref{eq:liftandiso}} \\
& = \iota_{\tilde{V}+\zeta}(E \circ \Phi^{-1*})^*\mu \\
%& = \iota_{\tilde{V}+\zeta}(E\circ\Phi^{-1*} \circ U)^*\mu & \text{letting }U:=\Phi^{*}\circ O'\circ \Phi^{-1*}\\
& =  \iota_{\tilde{V}+\zeta}(\Phi \circ E )^*\mu &  \text{since $\Phi$ is an isometry}\\
& = \iota_{\tilde{V}+\zeta}(E^* \circ \Phi^*\mu) \\
&  = s_{\Phi} \iota_{\tilde{V}+\zeta}(E^* \mu). & \text{since $\mu$ is isometry-invariant} \label{eq:lastline2}
\end{align}
In the third line, note that, if $\Phi$ is a sub-Riemannian isometry, then $H\circ \Phi^* = H$, and it follows that $E\circ\Phi^{-1*}= \Phi \circ E$. From \eqref{eq:lastline} and \eqref{eq:lastline2}, we have proved that
\begin{align}
(\phi^*\mu_{\Gamma'}^r)|_q(V) -
 s_\phi \mu_\Gamma^r|_q(V) & = s_\phi \int_{A_q \Gamma\cap \{\sqrt{2H}<r\}} \left(\iota_{\tilde{V}+\zeta}(E^* \mu)-\iota_{\tilde{V}}(E^* \mu)\right) \\
& = s_\phi \int_{A_q \Gamma\cap \{\sqrt{2H}<r\}} \iota_{\zeta}(E^* \mu), \label{eq:vanish}
\end{align}
where we crucially used, in the second line, that $\Gamma$ has dimension one so that $\iota_{\tilde{V} +\zeta} = \iota_{\tilde{V}} + \iota_{\zeta}$. We will prove that the r.h.s.\ of \eqref{eq:vanish} vanishes by exploiting symmetries of $M=\mathbb{H}_{2d+1}$.

Let $I:T^*M \to T^*M$ be the map $I(\lambda) := -\lambda$. Recall the definition of $\zeta$ from \eqref{eq:liftandiso}:
\begin{equation}
\zeta = \tilde{V}- (\Phi^*)_*\widetilde{\Phi_* V}.
\end{equation}
Note that if $\tilde{V}$ is a lift of $V$ on $A\Gamma$ then since $I:A\Gamma \to A\Gamma$ is a bundle map we have that $I_*\tilde{V}$ is a lift of $V$ on $A\Gamma$ and thus according to our notation $I_* \tilde{V} = \tilde{V}$. Furthermore, it also holds $I \circ \Phi^* = \Phi^* \circ I$. It follows that 
\begin{equation}\label{eq:1tocheck}
I_* \zeta = \zeta.
\end{equation}
We claim that if $\mu$ is the $n$-form inducing the Lebesgue measure, the following identity holds:
\begin{equation}\label{eq:2tocheck}
I^*  (E^* \mu)|_{T_q^* M} = (-1)^{n} (E^*\mu)|_{T_q^* M}.
\end{equation}
(Note that $I=i^*$, where $i:M\to M$ is $i(x,y,z)=(-x,-y,-z)$. If $i$ were a sub-Riemannian isometry, then we would have $i\circ E = E\circ i^{-1*} = E\circ I$, and \eqref{eq:2tocheck} would follow immediately, without the need of restriction to $T_q^*M$. However this is not the case and we have to argue differently.) By left-invariance of $E$ and $\mu$ it is sufficient to prove \eqref{eq:2tocheck} at $q=e$ (the identity of the Heisenberg group). We use coordinates $(x,z) \in \R^{2d}\times \R$ on $ M$, and dual coordinates $(p_x,p_z) \in \R^{2d}\times \R$ on each fiber $T^*_e M$, in terms of which it holds
\begin{equation}
E^*\mu|_{T_e^* M} = J(p_x,p_z) dp_{x_1}\wedge \dots \wedge dp_{x_{2d}} \wedge dp_z,
\end{equation}
where, as computed e.g.\ in \cite[Lemma 15]{R-MCP}, $J(p_x,p_z)$ is the Jacobian determinant
\begin{equation}
J(p_x,p_z) = \frac{2^{2d}}{p_z^{2d}}\|p_x\|^2 \sin\left(\frac{p_z}{2}\right)^{2d-1}\left(\sin\left(\frac{p_z}{2}\right)-\frac{p_z}{2}\cos\left(\frac{p_z}{2}\right)\right).
\end{equation}
Since $I(p_x,p_z) = (-p_x,-p_z)$ and $(p_x,p_z)\mapsto J(p_x,p_z)$ is even, then \eqref{eq:2tocheck} follows.

Therefore, for any tuple of vector fields $\xi = (\xi_1,\dots,\xi_{n-1})$ on $A_q \Gamma$ it holds
\begin{align}
I^* (\iota_\zeta (E^*\mu))(\xi) & = (E^*\mu)(\zeta,I_*\xi)\\
& =(E^*\mu)(\iota_*\zeta,I_*\xi) & \text{by \eqref{eq:1tocheck}} \\
& = I^* (E^*\mu)(\zeta,\xi) \\
& = (-1)^n E^*\mu (\zeta,\xi) & \text{by \eqref{eq:2tocheck}} \\
& = (-1)^n (\iota_\zeta (E^*\mu))(\xi)\label{eq:verylastline}.
\end{align}
Note that, in the above equality, it is crucial that we are acting on vector fields tangent to the fibers $A_q \Gamma$ in order to use \eqref{eq:2tocheck}. Thus, since $\dim(A_q \Gamma)=n-1$, then $I|_{A_q \Gamma} :A_q \Gamma \to A_q \Gamma$ is $(-1)^{n-1}$-orientation preserving. Hence, we have
\begin{align}
\int_{A_q \Gamma\cap \{\sqrt{2H}<r\}} \iota_{\zeta}(E^* \mu) 
& = \int_{I(A_q \Gamma\cap \{\sqrt{2H}<r\})} \iota_{\zeta}(E^* \mu), \\
& = (-1)^{n-1}\int_{A_q \Gamma\cap \{\sqrt{2H}<r\}} I^* (\iota_{\zeta}(E^* \mu)), \\
& = -\int_{A_q \Gamma\cap \{\sqrt{2H}<r\}} \iota_{\zeta}(E^* \mu), &  \text{by \eqref{eq:verylastline}}.
\end{align}
This completes the proof that $\phi^*\mu_\Gamma^r = s_{\phi}\mu_{\Gamma'}^r$.
\end{proof}

%\begin{rmk}
%It is easy to produce curves in $\mathbb{H}_{2d+1}$, with the same Riemannian length, whose tangent vector $\dot\gamma$ have different angles with respect to the Reeb field, that have different Weyl's tube coefficients (while of-course they are isometric as Riemannian manifolds). This shows that the condition of preserving the Reeb angle is also necessary for having the same Weyl's tube formula.
%\end{rmk}

\bibliographystyle{abbrv}
\bibliography{biblio.bib}
\end{document}